\documentclass{arthal}
\usepackage{amsmath,amssymb,graphicx,textcomp,enumerate,bbm,xcolor,latexsym,theorem}
\newcommand{\assign}{:=}
\newcommand{\cdummy}{\cdot}
\newcommand{\minusassign}{-\!\!=}
\newcommand{\nobracket}{}
\newcommand{\nocomma}{}
\newcommand{\noplus}{}
\newcommand{\nosymbol}{}
\newcommand{\tmabbr}[1]{#1}

\newcommand{\tmem}[1]{{\em #1\/}}
\newcommand{\tmmathbf}[1]{\ensuremath{\boldsymbol{#1}}}
\newcommand{\tmop}[1]{\ensuremath{\operatorname{#1}}}

\newcommand{\tmstrong}[1]{\textbf{#1}}
\newcommand{\tmtextbf}[1]{{\bfseries{#1}}}
\newcommand{\tmverbatim}[1]{{\ttfamily{#1}}}
\newenvironment{enumeratealpha}{\begin{enumerate}[a{\textup{)}}] }{\end{enumerate}}
\newenvironment{itemizeminus}{\begin{itemize} }{\end{itemize}}

\newcommand{\kk}{\ensuremath{\mathbbm{K}}}
\newcommand{\nw}{\ensuremath{\mathfrak{n}}}
\newcommand{\ZZ}{\ensuremath{\mathbbm{Z}}}
\newcommand{\Kx}{\ensuremath{\mathbbm{K}[\tmmathbf{x}]}}

\newcommand{\HO}[1]{[}

\headers{Fast algorithm for border bases of Artinian Gorenstein algebras}{B. Mourrain}

\begin{document}

\title{Fast algorithm for border bases of Artinian Gorenstein algebras}

\author{
Bernard Mourrain\thanks{UCA,
Inria M\'editerran\'ee,
\textsc{aromath},
Sophia
Antipolis,
France, \email{bernard.mourrain@inria.fr}}
}

\date{April 25, 2017}

\maketitle

\begin{abstract}
  Given a multi-index sequence $\sigma$, we present a new efficient algorithm
  to compute generators of the linear recurrence relations between the terms
  of $\sigma$. We transform this problem into an algebraic one, by identifying
  multi-index sequences, multivariate formal power series and linear
  functionals on the ring of multivariate polynomials. In this setting, the
  recurrence relations are the elements of the kernel $I_{\sigma}$ of the
  Hankel operator $H_{\sigma}$ associated to $\sigma$. We describe the
  correspondence between multi-index sequences with a Hankel operator of
  finite rank and Artinian Gorenstein Algebras. We show how the algebraic
  structure of the Artinian Gorenstein algebra $\mathcal{A}_{\sigma}$
  associated to the sequence $\sigma$ yields the structure of the terms
  $\sigma_{\alpha}$ for all $\alpha \in \mathbbm{N}^n$. This structure is
  explicitly given by a border basis of $\mathcal{A}_{\sigma}$, which is
  presented as a quotient of the polynomial ring $\mathbbm{K} [x_1, \ldots,
  x_n]$ by the kernel $I_{\sigma}$ of the Hankel operator $H_{\sigma}$. The
  algorithm provides generators of $I_{\sigma}$ constituting a border basis,
  pairwise orthogonal bases of $\mathcal{A}_{\sigma}$ and the tables of
  multiplication by the variables in these bases. It is an extension of
  Berlekamp-Massey-Sakata (BMS) algorithm, with improved complexity bounds. We
  present applications of the method to different problems such as the
  decomposition of functions into weighted sums of exponential functions,
  sparse interpolation, fast decoding of algebraic codes, computing the
  vanishing ideal of points, and tensor decomposition. Some benchmarks
  illustrate the practical behavior of the algorithm.
\end{abstract}

\section{Introduction}

Discovering hidden structures from probing or sampling is a problem which
appears in many contexts and in many applications. An interesting instance of
this general problem is recovering the structure of a sequence of values, from
the knowledge of some of its terms. It consists in guessing any term of the
sequence from the first known terms. A classical way to tackle this problem,
which goes back to Bernoulli, is to find linear recurrence relations between
the first terms of a sequence, to compute the roots of the associated
characteristic polynomial and to deduce the expression of any term of the
sequence from these roots.

In this paper, we consider the structure discovering problem for
multi-index sequences
$\sigma = (\sigma_{\alpha})_{\alpha \in \mathbbm{N}^n} \in
\mathbbm{K}^{\mathbbm{N}^n}$
of values in a field $\mathbbm{K}$.  Given a finite set of values
$\sigma_{\alpha}$ for $\alpha \in \tmmathbf{a} \subset \mathbbm{N}^n$,
we want to guess a formula for the general terms of the sequence
$\sigma$.  An important step of this approach is to compute
characteristic polynomials of the sequence
$\sigma = (\sigma_{\alpha})_{\alpha \in \mathbbm{N}^n}$.  They
correspond to multi-index {\tmem{recurrence relations}} with constant
coefficients between the terms of $\sigma$.  The ideal of these
recurrence relation polynomials define an Artinian Gorenstein algebra.
We present a fast algorithm to compute a border basis of this ideal
from the first terms of the sequence $\sigma$. This method also yields
a basis of the Artinian Gorenstein algebra as well as its
multiplicative structure. 

\paragraph{Related works} The approach that we present is related to Prony's
method in the univariate case and to its variants
{\cite{swindlehurst_performance_1992}}, {\cite{roy_esprit-estimation_1989}},
{\cite{golub_separable_2003}}, {\cite{beylkin_approximation_2005}}, ... and to
the more recent extensions in the multivariate case
{\cite{andersson_nonlinear_2010}}, {\cite{potts_parameter_2013}},
{\cite{kunis_multivariate_2016}}, {\cite{sauer_pronys_2016-1}}. Linear algebra
tools are used to determine a basis of the quotient algebra
$\mathcal{A}_{\sigma}$ or to compute an $H$-basis for the presentation of
$\mathcal{A}_{\sigma}$. An analysis of the complexity of these approaches
yields bounds in $\mathcal{O} (\tilde{s}^3)$ (or $\mathcal{O}
(\tilde{s}^{\omega})$), where $\tilde{s}$ is the size of the Hankel matrices
involved in these methods, typically the number $\binom{d' + n}{n}$ of
monomials in degree at most $d' < \frac{d}{2}$ if the terms of the sequence
are known up to the degree $d$. The problem is also related to Pad{\'e}
approximants, well investigated in the univariate case
{\cite{brent_fast_1980}}, {\cite{beckermann_uniform_1994}},
{\cite{gathen_modern_2013}}, but much less developed in the multivariate case
{\cite{power_finite_1982}}, {\cite{cuyt_how_1999}}.

Finding recurrence relations is a problem which is also well developed in the
univariate case. Berlekamp {\cite{berlekamp_nonbinary_1968}} and Massey
{\cite{massey_shift-register_1969}} proposed an efficient algorithm to compute
such recurrence relations, with a complexity in $\mathcal{O} (r^2)$ where $r$
is the size of the minimal recurrence relation. Exploiting further the
properties of Hankel matrices, the complexity of computing recurrence
relations can be reduced in the univariate case to $\tilde{\mathcal{O}}(r)$.

Sakata extended Berlekamp-Massey algorithm to
multi-index sequences, computing a Gr{\"o}bner
basis of the polynomials in the kernel of a multi-index Hankel
matrix {\cite{sakata_finding_1988}}. See also
{\cite{saints_algebraic-geometric_1995}} for an analysis and overview of the
algorithm. The computation of multivariate linear recurrence relations have
been further investigated, e.g. in {\cite{fitzpatrick_finding_1990}} and more
recently in {\cite{berthomieu_linear_2015}}, where the coefficients of the
Gr{\"o}bner basis are computed by solving multi-index Hankel systems.

\paragraph{Contributions} We translate the structure discovering
problem into an algebraic setting, by identifying multi-index
sequences of values, generating formal power series and linear
functionals on the ring of polynomials. Through this identification,
we associate to a multi-index sequence ${\sigma}$, a Hankel operator
$H_{\sigma}$ which kernel $I_{\sigma}$ defines an Artinian Gorenstein Algebra
$\mathcal{A}_{\sigma}$ when $H_{\sigma}$ is of finite rank.  We present a new efficient algorithm to
compute the algebraic structure of $\mathcal{A}_{\sigma}$, using the first terms $\sigma_{\alpha}$ for
$\alpha \in \tmmathbf{a} \subset \mathbbm{N}^n$. The structure
$\mathcal{A}_{\sigma}$ is described by a border basis of the ideal
$I_{\sigma}$.

This algorithm is an extension of the Berlekamp-Massey-Sakata (BMS)
algorithm.  It computes border bases of the recurrence relations,
which are more general than Gr{\"o}bner bases. They also offer a
better numerical stability \cite{mourrain_stable_2008} in the solving
steps required to address the decomposition problem. The algorithm,
based on a Gram-Schmidt orthogonalisation process, is simplified. The
complexity bound also improves the previously known bounds for
computing such recurrence relations. We show that the arithmetic
complexity of computing a border basis is in
$\mathcal{O} ((r + \delta) r s)$ where $r$ is the number of roots of
$I_{\sigma}$ (counted with multiplicities), $\delta$ is the size of
the border of the monomial basis and $s$ is the number of known terms
of the sequence $\sigma$.

The algorithm outputs generators of the recurrence relations, a
monomial basis, an orthogonal basis and the tables of multiplication
by the variables in this basis of $\mathcal{A}_{\sigma}$.  The
structure of the terms of the sequence $\sigma$ can be deduced from this output, by
applying classical techniques for solving polynomial systems from
tables of multiplication. We show how the algorithm can be applied to
different problems such as the decomposition of functions into
weighted sums of exponential functions, sparse interpolation, fast
decoding of algebraic codes, vanishing ideal of points, and tensor
decomposition.

\paragraph{Notation} Let $\mathbbm{K}$ be a field, $\bar{\mathbbm{K}}$ its
algebraic closure, $\mathbbm{K} [x_1, \ldots, x_n] =\mathbbm{K} [\tmmathbf{x}]$
be the ring of polynomials in the variables $x_1, \ldots, x_n$ with
coefficients in the field $\mathbbm{K}$, $\mathbbm{K} [[y_1, $ $\ldots, y_n]]
=\mathbbm{K} [\tmmathbf{y}]$ be the ring of formal power series in the
variables $y_1, \ldots, y_n$ with coefficients in $\mathbbm{K}$. We denote by
$\mathbbm{K}^{\mathbbm{N}^n}$ the set of sequences $\sigma =
(\sigma_{\alpha})_{\alpha \in \mathbbm{N}^n}$ \ of numbers $\sigma_{\alpha}
\in \mathbbm{K}$, indexed by $\mathbbm{N}^n$. $\forall \alpha = (\alpha_1,
\ldots, \alpha_n) \in \mathbbm{N}^n,$ $\alpha ! = \prod_{i =
1}^n \alpha_i !$, $\tmmathbf{x}^{\alpha} = \prod_{i = 1}^n x_i^{\alpha_i} $.
The monomials in $\mathbbm{K} [\tmmathbf{x}]$ are the elements of the form
$\tmmathbf{x}^{\alpha}$ for $\alpha \in \mathbbm{N}^n$. For a set $B \subset
\mathbbm{K} [\tmmathbf{x}]$, $B^+ = \cup_{i = 1}^n x_i B \cup B$, $\partial B
= B^+ \setminus B$. A set $B$ of monomials of $\mathbbm{K} [\tmmathbf{x}]$ is
connected to $1$, if $1 \in B$ and for $\tmmathbf{x}^{\beta} \in B$ different
from $1$, there exists $\tmmathbf{x}^{\beta'} \in B$ and $i \in [1, n]$ such
that $\tmmathbf{x}^{\beta} = x_i \tmmathbf{x}^{\beta'}$. For $F \subset
\mathbbm{K} [\tmmathbf{x}]$, $\langle F \rangle$ is the vector space of
$\mathbbm{K} [\tmmathbf{x}]$ spanned by $F$ and $(F)$ is the ideal generated
by $F$. For $V,V'\subset \mathbbm{K} [\tmmathbf{x}]$, $V\cdummy V'$ is
the set of products of an element of $V$ by an element of $V'$.

\section{ Polynomial-Exponential series}

In this section, we recall the correspondence between sequences $\sigma =
(\sigma_{\alpha})_{\alpha \in \mathbbm{N}^n} \in \mathbbm{K}^{\mathbbm{N}^n}$
associated to polynomial-exponential series and Artinian Gorenstein
Algebras.

\subsection{Duality}

A sequence $\sigma = (\sigma_{\alpha})_{\alpha \in \mathbbm{N}^n} \in
\mathbbm{K}^{\mathbbm{N}^n}$ is naturally associated to a linear form operating
on polynomials, that is, an element of $\tmop{Hom}_{\mathbbm{K}} (\mathbbm{K}
[\tmmathbf{x}], \mathbbm{K}) = \Kx^{\ast}$, as follows:
\[ p = \sum_{\alpha \in A \subset \mathbbm{N}^n} p_{\alpha}
   \tmmathbf{x}^{\alpha} \in \mathbbm{K} [\tmmathbf{x}] \mapsto \langle \sigma
   \mid p \rangle = \sum_{\alpha \in A \subset \mathbbm{N}^n} p_{\alpha}
   \sigma_{\alpha} . \]
This correspondence is bijective since a linear form $\sigma \in \Kx^{\ast}$ is
uniquely defined by the sequence $\sigma_{\alpha} = \langle \sigma \mid
\tmmathbf{x}^{\alpha} \rangle$ for $\alpha \in \mathbbm{N}^n$. The
coefficients\tmtextbf{} $\sigma_{\alpha} = \langle \sigma \mid
\tmmathbf{x}^{\alpha} \rangle$ for $\alpha \in \mathbbm{N}^n$ are also called
the {\tmem{moments}} of $\sigma$. Hereafter, we will identify
$\mathbbm{K}^{\mathbbm{N}^n}$ with $\Kx^{\ast} = \tmop{Hom}_{\mathbbm{K}}
(\mathbbm{K} [\tmmathbf{x}], \mathbbm{K})$.

The dual space $\Kx^{\ast}$ has a natural structure of $\Kx$-module, defined as
follows: $\forall \sigma \in \Kx^{\ast}, \forall p, q \in \Kx$,
\begin{eqnarray*}
  \langle p \star \sigma \mid q \rangle & = & \langle \sigma
  \mid p q \rangle .
\end{eqnarray*}
We check that $\forall \sigma \in \Kx^{\ast}, \forall p
\nocomma, q \in \Kx$, $(p q) \star \sigma = p \star (q \star \sigma)$.
See e.g. {\cite{emsalem_geometrie_1978}}, {\cite{mourrain_isolated_1996}} for
more details.

For any $\sigma \in \Kx^{\ast}$, the inner product associated to $\sigma$ on
$\mathbbm{K} [\tmmathbf{x}]$ is defined as follows:
\begin{eqnarray*}
  \Kx \times \Kx & \rightarrow & \mathbbm{K}\\
  (p, q) & \mapsto & \langle p, \nosymbol \nosymbol q \rangle_{\sigma} \assign
  \langle \sigma | \nobracket p \nosymbol \nosymbol q \rangle .
\end{eqnarray*}
Sequences in $\mathbbm{K}^{\mathbbm{N}^n}$ are also in correspondence with
series in $\mathbbm{K} [[\tmmathbf{z}]]$, via the so-called
$\tmmathbf{z}$-transform:
\[ \sigma = (\sigma_{\alpha})_{\alpha \in \mathbbm{N}^n} \in
   \mathbbm{K}^{\mathbbm{N}^n} \mapsto \sigma (\tmmathbf{z}) = \sum_{\alpha
   \in \mathbbm{N}^n} \sigma_{\alpha} \tmmathbf{z}^{\alpha} \in \mathbbm{K}
   [[\tmmathbf{z}]] . \]
If $\mathbbm{K}$ is a field of characteristic $0$, we can identify the
sequence $\sigma = (\sigma_{\alpha})_{\alpha \in \mathbbm{N}^n} \in
\mathbbm{K}^{\mathbbm{N}^n}$with the series $\sigma (\tmmathbf{y}) =
\sum_{\alpha \in \mathbbm{N}^n} \sigma_{\alpha} 
\frac{\tmmathbf{y}^{\alpha}}{\alpha !} \in \mathbbm{K} [[\tmmathbf{y}]] .$
Using this identification, we have \ $\forall p \in \mathbbm{K}
[\tmmathbf{x}], \forall \sigma \in \mathbbm{K} [[\tmmathbf{y}]]$, $p \star
\sigma (\tmmathbf{y} \nosymbol \nosymbol) = p (\partial_{y_1}, \ldots,
\partial_{y_n}) (\sigma (\tmmathbf{y}))$.

Through these identifications, the dual basis of the monomial basis
$(\tmmathbf{x}^{\alpha})_{\alpha \in \mathbbm{N}^n}$ is
$(\tmmathbf{z}^{\alpha})_{\alpha \in \mathbbm{N}^n}$ in $\mathbbm{K}
[[\tmmathbf{z}]]$ and $\left( \frac{\tmmathbf{y}^{\alpha}}{\alpha !}
\right)_{\alpha \in \mathbbm{N}^n}$ in $\mathbbm{K} [[\tmmathbf{y}]]$.

Among the elements of $\tmop{Hom} (\mathbbm{K} [\tmmathbf{x}],
\mathbbm{K})$, we have the evaluation $\tmmathbf{e}_{_{\xi}} : p
(\tmmathbf{x}) \in \mathbbm{K} [\tmmathbf{x}] \mapsto p (\xi) \in
\mathbbm{K}$ at a point $\xi \in \mathbbm{K}^n$, which corresponds to the
sequence $(\xi^{\alpha})_{\alpha \in \mathbbm{N}^n}$ or to the series
$\tmmathbf{e}_{_{\xi}} (\tmmathbf{z}) = \sum_{\alpha \in \mathbbm{N}^n}
\xi^{\alpha} \tmmathbf{z}^{\alpha} = \hspace{0.25em}  \frac{1
\hspace{0.25em}}{\prod_{i = 1}^n (1 - \xi_i z_i)} \in \mathbbm{K}
[[\tmmathbf{z}]]$, or to the series $\tmmathbf{e}_{_{\xi}} (\tmmathbf{y}) =
\sum_{\alpha \in \mathbbm{N}^n} \xi^{\alpha} 
\frac{\tmmathbf{y}^{\alpha}}{\alpha !} = \hspace{0.25em} e^{\xi_1 y_1 + \cdots
+ \xi_n y_n} = e^{\langle \xi, \mathbf{y} \rangle}$ in $\mathbbm{K}
[[\tmmathbf{y}]] $. These series belong to the more general family
$\mathcal{{POLYEXP}}$ of polynomial-exponential series $\sigma = \sum_{i
= 1}^r \omega_i \tmmathbf{e}_{\xi_i} \in
\mathbbm{K} [[\tmmathbf{y}]]$ with $\xi_i \in \mathbbm{K}^n, \omega_i
\in \mathbbm{K} [\tmmathbf{y}]$. This set corresponds in
$\mathbbm{K} [[\tmmathbf{z}]]$ to the set of series of the form
\[ \sigma = \sum_{i = 1}^r \sum_{\alpha \in A_i} \frac{\omega_{i, \alpha}
   \tmmathbf{z}^{\alpha} \hspace{0.25em}}{\prod_{j = 1}^n (1 - \xi_{i, j}
   z_j)^{1 + \alpha_j}} \]
with $\xi_i \in \mathbbm{K}^n, \omega_{i, \alpha} \in \mathbbm{K}
\nocomma, \alpha \in A_i \subset \mathbbm{N}^n$ and $A_i$ finite.

\begin{definition}
  For a subset $D \subset \mathbbm{K} [[\tmmathbf{y}]]$, the {\tmem{inverse
      system}} generated by $D$ is the vector space spanned by the elements
  $p \star \delta$ for $\delta \in  D$, $p \in \mathbbm{K} [\tmmathbf{x}]$, that is, by the
  elements in $D$ and all their derivatives. \
  
  For $\omega \in \mathbbm{K} [\tmmathbf{y}]$, we denote by ${\mu}
  (\omega)$ the dimension of the inverse system of $\omega$, generated by $\omega$ and
  all its derivatives. For $\sigma = \sum_{i = 1}^r \omega_i
  \tmmathbf{e}_{\xi_i} \in \mathcal{{POLYEXP}} (\tmmathbf{y})$, ${\mu} (\sigma) = \sum_{i =
  1}^r {\mu} (\omega_i)$.
\end{definition}

\subsection{Hankel operators}

The external product $\star$ allows us to define a
Hankel operator as a multiplication operator by a dual element $\in
\mathbbm{K} [\tmmathbf{x}]^{\ast}$: 

\begin{definition}
  The Hankel operator associated to an element $\sigma \in \mathbbm{K}
  [\tmmathbf{x}]^{\ast} =\mathbbm{K}^{\mathbbm{N}^n}$ is
  \begin{eqnarray*}
    H_{\sigma} : \mathbbm{K} [\tmmathbf{x}] & \rightarrow &
    \mathbbm{K} [\tmmathbf{x}]^{\ast}\\
    p = \sum_{\beta \in B} p_{\beta} \tmmathbf{x}^{\beta} & \mapsto & p \star
    \sigma = \left( \sum_{\beta \in B} p_{\beta} \sigma_{\alpha + \beta}
    \right)_{\alpha \in \mathbbm{N}^n} .
  \end{eqnarray*}
  Its kernel is denoted $I_{\sigma}$. We say that the series $\sigma$ has 
  finite rank $r \in \mathbbm{N}$ if $\tmop{rank} H_{\sigma} = r < \infty$.
\end{definition}

As $\forall p, q \in \mathbbm{K} [\tmmathbf{x}]$, \ $p q \star \sigma = p
\star (q \star \sigma)$, we easily check that $I_{\sigma} = \ker H_{\sigma}$
is an ideal of $\mathbbm{K} [\tmmathbf{x} \mathbf{}]$ and that
$\mathcal{A}_{\sigma} =\mathbbm{K} [\tmmathbf{x}] / I_{\sigma}$ is an
algebra.

Given a sequence $\sigma = (\sigma_{\alpha})_{\alpha \in \mathbbm{N}^n} \in
\mathbbm{K}^{\mathbbm{N}^n}$, the kernel of $H_{\sigma}$ is the set of
polynomials $p = \sum_{\beta \in B} p_{\beta} \tmmathbf{x}^{\beta}$ such that
$\sum_{\beta \in B} p_{\beta} \sigma_{\alpha + \beta}=0$ for all $\alpha \in
\mathbbm{N}^n$. This kernel is the set of {\tmem{linear recurrence relations}}
of the sequence $\sigma = (\sigma_{\alpha})_{\alpha \in \mathbbm{N}^n}$.

\begin{remark}
  The matrix of the operator $H_{\sigma}$ in the bases
  $(\tmmathbf{x}^{\alpha})_{\alpha \in \mathbbm{N}^n}$ and its dual basis
  $(\tmmathbf{z}^{\alpha})_{\alpha \in \mathbbm{N}^n}$ is
  \begin{eqnarray*}
    {}[H_{\sigma}] & = & (\sigma_{\alpha + \beta})_{\alpha, \beta \in
    \mathbbm{N}^n} = (\langle \sigma |  \nobracket
    \tmmathbf{x}^{\alpha + \beta}_{\nocomma} \rangle)_{\alpha, \beta \in
    \mathbbm{N}^n} .
  \end{eqnarray*}
\end{remark}

The coefficients of $[H_{\sigma}]$ depend only the sum of the multi-indices
indexing the rows and columns, which explains why it is called a
{\tmem{Hankel}} operator.

\

In the reconstruction problem, we are dealing with truncated
series with known coefficients $\sigma_{\alpha}$ for $\alpha$ in
a subset $\tmmathbf{a}$ of $\mathbbm{N}^n$. This leads to the definition of
truncated Hankel operators.

\begin{definition}
  For \ two vector spaces $V, V' \subset \mathbbm{K} [\tmmathbf{x}]$ and \
  $\sigma \in \langle V \cdummy V' \rangle^{\ast} \subset \mathbbm{K}
  [\tmmathbf{x}]^{\ast}$, the {\tmem{truncated Hankel operator}} on $(V, V')$,
  denoted by $H_{\sigma}^{V, V'}$, is the following map:
  \begin{eqnarray*}
    H^{V, V'}_{\sigma} : V & \rightarrow & V'^{\ast} =
    \tmop{Hom}_{\mathbbm{K}} (V', \mathbbm{K})\\
    p (\tmmathbf{x}) & \mapsto & p \star \sigma_{| V' \nobracket} .
  \end{eqnarray*}
\end{definition}

If $B = \{ b_1, \ldots, b_r \}$ (resp. $B' = \{ b_1', \ldots
\nocomma, b_r' \}$) is a basis of $V$ (resp. $V'$), then the matrix of the
operator $H_{\sigma}^{V, V'}$ in $B$ and the dual basis of $B'$ is
\[ [H_{\sigma}^{B, B'}] = (\langle \sigma |  \nobracket b_j b_i'  \rangle)_{1
   \leqslant i, j \leqslant r} . \]
If $B$ and $B'$ are monomial sets, we obtain the so-called {\tmem{truncated
moment matrix}} of $\sigma$:
\begin{eqnarray*}
  {}[H_{\sigma}^{B, B'}] & = & (\sigma_{\beta + \beta'})_{\beta' \in B', \beta
  \in B}
\end{eqnarray*}
(identifying a monomial $\tmmathbf{x}^{\beta}$ with its exponent $\beta$).
These structured matrices share with the classical univariate Hankel matrices
many interesting properties (see e.g. in {\cite{mourrain_multivariate_2000}}).

\subsection{Artinian Gorenstein algebra}

A $\mathbbm{K}$-algebra $\mathcal{A}$ is {\tmem{Artinian}} if
$\dim_{\mathbbm{K}} (\mathcal{A}) < \infty$.  It can be represented as the
quotient $\mathcal{A}=\mathbbm{K} [\tmmathbf{x}] / I$ of a polynomial ring
$\mathbbm{K} [\tmmathbf{x}]$ by a (zero-dimension) ideal $I \subset
\mathbbm{K} [\tmmathbf{x}]$. 

A classical result states that the quotient algebra $\mathcal{A} =\mathbbm{K}
[\tmmathbf{x}] / I$ is finite dimensional, i.e. Artinian, iff
$\mathcal{V}_{\bar{\mathbbm{K}}} (I)$ is finite, that is, $I$ defines a finite
number of (isolated) points in $\bar{\mathbbm{K}}^n$ (see e.g.
{\cite{cox_ideals_2015}}[Theorem 6] or
{\cite{elkadi_introduction_2007}}[Theorem 4.3]).

The dual $\mathcal{A}^{\ast} = \tmop{Hom}_{\mathbbm{K}} (\mathcal{A},
\mathbbm{K})$ of $\mathcal{A}=\mathbbm{K} [\tmmathbf{x}] / I$ is naturally
identified with the sub-space
\[ I^{\bot} = \{ \sigma \in \mathbbm{K} [\tmmathbf{x}]^{\ast} \mid \forall p
   \in I, \langle \sigma | \nobracket p \rangle = 0 \} . \]
A {\tmem{Gorenstein}} algebra is defined as follows:

\begin{definition}
  A $\mathbbm{K}$-algebra $\mathcal{A}$ is Gorenstein if $\exists \sigma \in
  \mathcal{A}^{\ast} = \tmop{Hom}_{\mathbbm{K}} (\mathcal{A}, \mathbbm{K})$
  such that $\forall \rho \in \mathcal{A}^{\ast}, \exists a \in
  \mathcal{A}$ with $\rho = a \star \sigma$ and $a \star \sigma = 0$ implies
  $a = 0$.
\end{definition}

In other words, $\mathcal{A}=\mathbbm{K} [\tmmathbf{x}] / I$ is Gorenstein iff
$\mathcal{A}^{\ast} = \{ p \star \sigma \mid p \in \mathbbm{K} [\tmmathbf{x}]
\} = \tmop{im} H_{\sigma}$ and $p \star \sigma = 0$ implies $p \in I$.
Equivalently, $\mathcal{A}=\mathbbm{K} [\tmmathbf{x}] / I$ is Gorenstein iff
there exists $\sigma \in \mathbbm{K} [\tmmathbf{x}]^{\ast}$ such that we have
the exact sequence:
\begin{equation}
  0 \rightarrow I \rightarrow \mathbbm{K} [\tmmathbf{x}]
  \xrightarrow{H_{\sigma}} \mathcal{A}^{\ast} \rightarrow 0 \label{eq:seq}
\end{equation}
so that $H_{\sigma}$ induces an isomorphism between $\mathcal{A}=\mathbbm{K}
[\tmmathbf{x}] / I$ and $\mathcal{A}^{\ast}$.
In other words, a Gorenstein algebra $\mathcal{A}$ is the quotient of a polynomial
ring by the kernel of a Hankel operator, or equivalently by an ideal of recurrence relations of a
multi-index sequence.

An Artinian Gorenstein can thus be described by an element $\sigma \in \mathbbm{K} [\tmmathbf{x}]^{\ast}$, such that
$\tmop{rank} H_{\sigma} = \dim \mathcal{A}^{\ast} = \dim \mathcal{A}$ is
finite. In the following, we will assume that the Artinian Gorenstein algebra is
given by such an element $\sigma \in \mathbbm{K} [\tmmathbf{x}]^{\ast} \equiv
\mathbbm{K}^{\mathbbm{N}^n}$. The corresponding algebra will be
$\mathcal{A}_{\sigma} =\mathbbm{K} [\tmmathbf{x}] / I_{\sigma}$ where
$I_{\sigma} = \ker H_{\sigma}$.

By a multivariate generalization of Kronecker's theorem
{\cite{mourrain_polynomial-exponential_2016}}[Theorem 3.1], the sequences
$\sigma$ such that $\tmop{rank} H_{\sigma} = r < \infty$ are the
polynomial-exponential series $\sigma \in \mathcal{{POLYEXP}}$ with
${\mu} (\sigma) = r$.

The aim of the method we are presenting, is to compute the structure of the
Artinian Gorenstein algebra $\mathcal{A}_{\sigma}$ from the first terms of the
sequence $\sigma = (\sigma_{\alpha})_{\alpha \in \mathbbm{N}^n}$. We are going
to determine bases of $\mathcal{A}_{\sigma}$ and generators of the ideal
$I_{\sigma}$, from which we can deduce directly the multiplicative structure
of $\mathcal{A}_{\sigma}$.

The following lemma gives a simple way to test the linear independency in
$\mathcal{A}_{\sigma}$ using truncated Hankel matrices (see
{\cite{mourrain_polynomial-exponential_2016}}[Lemma 3.3]):

\begin{lemma}
  \label{lem:basis} Let $\sigma\in \mathbbm{K} [\tmmathbf{x}]^{\ast}$, $B = \{ b_1, \ldots, b_r \}$, $B' = \{ b_1',\ldots,$ $
  b_r' \} \subset \mathbbm{K} [\tmmathbf{x}]$. If the matrix $H_{\sigma}^{B,
  B'} = (\langle \sigma |  \nobracket b_i b_j' \rangle)_{\in B, \beta' \in
  B'}$ is invertible, then $B$ (resp. $B'$) is linearly independent in
  $\mathcal{A}_{\sigma}$.
\end{lemma}

This lemma implies that if $\dim \mathcal{A}_{\sigma} = r < + \infty$, $| B |
= | B' | = r = \dim \mathcal{A}_{\sigma}$ and $H_{\sigma}^{B, B'}$ is
invertible, then $(\tmmathbf{x}^{\beta})_{\beta \in B}$ and
$(\tmmathbf{x}^{\beta'})_{\beta' \in B'}$ are bases of $\mathcal{A}_{\sigma}$.

Given a Hankel operator $H_{\sigma}$ of finite rank $r$, it is clear that
the truncated operators will have at most rank $r$. We are going to use the
so-called {\tmem{flat extension}} property, which gives conditions under which
a truncated Hankel operator of rank $r$ can be extended to a Hankel operator
of the same rank (see {\cite{laurent_generalized_2009}} and extensions
{\cite{brachat_symmetric_2010}}, {\cite{bernardi_general_2013}},
{\cite{mourrain_polynomial-exponential_2016}}).

\begin{theorem}
  \label{thm:flatext}Let $V, V' \subset \mathbbm{K} [\tmmathbf{x}]$ be vector
  spaces connected to $1$, \tmtextbf{} such that $x_1, \ldots, x_n \in V$ and
  let $\sigma \in \langle V \cdummy V' \rangle^{\ast}$. Let $B \subset
  V$, $B' \subset V'$ such that $B^+ \subset V, B'^+ \subset V'$. If
  $\tmop{rank} H_{\sigma}^{V, V'} = \tmop{rank} H_{\sigma}^{B, B'} = r$, then
  there is a unique extension $\tilde{\sigma} \in \mathbbm{K}
  [[\tmmathbf{y}]]$ such that $\tilde{\sigma}$ coincides with $\sigma$ on
  $\langle V \cdot V' \rangle$ and $\tmop{rank} H_{\tilde{\sigma}} = r$. In
  this case, $\tilde{\sigma} \in \mathcal{{POLYEXP}}$ with $r ={\mu}
  (\sigma)$ and $I_{\tilde{\sigma}} = (\ker H_{\sigma}^{B^+, B'})$.
\end{theorem}

\subsection{Border bases}

We recall briefly the definition of border basis and the main properties, that
we will need. Let $B$ be a monomial set of $\mathbbm{K} [\tmmathbf{x}]$.

\begin{definition}
  A rewriting family $F$ for a (monomial) set $B$ is a set of polynomials $F =
  \{ f_i \}_{i \in \tmmathbf{i}} \subset \mathbbm{K} [\tmmathbf{x}]$ such that
  $f_i =\tmmathbf{x}^{\alpha_i} + b_i$ with $b_i \in \langle B \rangle$,
  $\alpha_i \in \partial B$, $\alpha_i \neq \alpha_j$ if $i \neq j$. The
  rewriting family $f$ is complete if $(\tmmathbf{x}^{\alpha_i})_{i \in
  \tmmathbf{i}} = \partial B$.
\end{definition}

The monomial $\tmmathbf{x}^{\alpha_i}$ is called the leading
monomial of $f_i$ and denoted $\gamma (f_i) .$

\begin{definition}
  A family $F \subset \mathbbm{K} [\tmmathbf{x}]$ is a border basis with
  respect to $B$ if it is a complete rewriting family for $B$ such that
  $\mathbbm{K} [\tmmathbf{x}] = \langle B \rangle \oplus (F)$.
\end{definition}

This means that any element of $\mathbbm{K} [\tmmathbf{x}]$ can be projected
along the ideal $I=(F)$ onto a unique element of $\langle B \rangle$. In other
words, $B$ is a basis of the quotient algebra $\mathcal{A}=\mathbbm{K}
[\tmmathbf{x}] / I$.

Let $B^{[0]} = B$ and for $k \in \mathbbm{N}$, $B^{[k + 1]} = (B^{[k]})^+$. If
$1 \in B$, then for any $p \in \mathbbm{K} [\tmmathbf{x}]$, there exist $k \in
\mathbbm{N}$, such that $p \in \langle B^{[k]} \rangle$.

For a complete rewriting family $F$ with respect to a monomial set $B$
containing $1$, a projection $\pi_F$ of $\mathbbm{K} [\tmmathbf{x}]$ on
$\langle B \rangle$ can be defined recursively on the set of monomials $m$ of
$\mathbbm{K} [\tmmathbf{x}]$ by
\begin{itemize}
  \item if $m \in B$, $\pi_F (m) = m$;
  
  \item if $m \in \partial B$, $\pi_F (m) = m - f$ where $f$ is the (unique)
  polynomial in $F$ for which $\gamma (f) = m$,
  
  \item if $m \in B^{[k + 1]} - B^{[k]}$ for $k > 1$, there exists $m' \in
  B^{[k]}$ and $i_0 \in [1, n]$ such that $m = x_{i_0} m'$. Let $\pi_F (m) =
  \pi_F (x_{i_0} \pi_F (m'))$.
\end{itemize}
This map defines a projector from $\mathbbm{K} [\tmmathbf{x}]$ onto $\langle B
\rangle$. The kernel of $\pi_F$ is contained in the ideal $(F)$. The family
$F$ is a border basis iff $\ker (\pi_F) = (F)$.

Checking that a complete rewriting family is a border basis reduces to
checking commutation properties. This leads to efficient algorithms to compute
a border basis. For more details, see {\cite{mourrain_generalized_2005}},
{\cite{mourrain_stable_2008}}, {\cite{mourrain_border_2012}}.

A special case of border basis is when the leading term $\gamma (f)$ of $f
\in F$ is the maximal monomial of $f$ for a monomial ordering $\succ$. Then
$F$ is a Gr{\"o}bner basis of $I$ for this monomial ordering $\succ$.

A border basis $F$ with respect to a monomial set $B$ gives directly the
tables of multiplication $M_i$ by the variables $x_i$ in the basis $B$. For a
monomial $b \in B$, $M_i (b) = \pi_F (x_i b) = x_i b - f$ with$f \in F$ such
that $\gamma (f) = x_i b$ if $x_i b \in \partial B$ and $f = 0$ otherwise.

\section{Border bases of series}\label{sec:5}

Given the first terms $\sigma_{\alpha}$ for $\alpha \in \tmmathbf{a}$ of the
sequence $\sigma = (\sigma_{\alpha})_{\alpha \in \mathbbm{N}^n} \in
\mathbbm{K}^{\mathbbm{N}^n}$, where $\tmmathbf{a} \subset \mathbbm{N}^n$ is a
finite set of exponents, we are going to compute a basis of
$\mathcal{A}_{\sigma}$ and generators of $I_{\sigma}$. We assume that the
monomial set $\tmmathbf{x}^{\tmmathbf{a}}=\{\tmmathbf{x}^{\alpha}, \alpha \in \tmmathbf{a}\}$ is connected to 1.

\subsection{Orthogonal bases of $\mathcal{A}_{\sigma}$}

An important step in the decomposition method consists in computing a basis
$B$ of $\mathcal{A}_{\sigma}$. In this section, we describe how to compute a
monomial basis $B = \{ \tmmathbf{x}^{\beta} \}$ and two other bases
$\tmmathbf{p}= (p_{\beta})$ and $\tmmathbf{q}= (q_{\beta})$, which are
pairwise orthogonal for the inner product $\langle \cdummy, \cdummy
\rangle_{\sigma}$:
\[ \langle p_{\beta}, q_{\beta'} \rangle_{\sigma} = \left\{ \begin{array}{ll}
     1 & \tmop{if} \beta = \beta'\\
     0 & \tmop{otherwise} .
   \end{array} \right. \]
To compute these pairwise orthogonal bases, we will use a projection process,
similar to Gram-Schmidt orthogonalization process. The main difference is that
we compute pairs $p_{\beta}, q_{\beta}$ of orthogonal polynomials. As
the inner product $\langle \cdummy, \cdummy \rangle_{\sigma}$ may be
isotropic, the two polynomials $p_{\beta}, q_{\beta}$ may not be
equal, up to a scalar. For a polynomial $f$ and two families of polynomials
$\tmmathbf{p}= [p_1, \ldots, p_l]$, $\tmmathbf{m}= [m_1, \ldots, m_l]$, we
will use the following procedure $\tmop{proj} (f, \tmmathbf{p},
\tmmathbf{m})$.

{\begin{algorithm}[H]\caption{\label{algo:proj}Orthogonal projection}{\tmstrong{Input:} $f \in \mathbbm{K}
[\tmmathbf{x}]$, $\tmmathbf{p}= [p_1, \ldots, p_l]$ and $\tmmathbf{m}= [m_1,
\ldots, m_l]$ such that $\langle p_i, m_j \rangle_{\sigma} = 0$ if $j < i$ \
and $\langle p_i, m_i \rangle_{\sigma} = 1$.
\begin{itemizeminus}
  \item $g = f$;
  
  \item for $i$ in $1 \ldots l$ do $g \,\minusassign \langle g, m_i
  \rangle_{\sigma} p_i$;
\end{itemizeminus}
{\tmstrong{Output:}} $g \assign \tmop{proj} (f, \tmmathbf{p}, \tmmathbf{m})$}
\end{algorithm}}
% \negskip

Algorithm \ref{algo:proj} corresponds to the Modified Gram-\-Schmidt
algorithm, when the scalar product is definite positive. It is known
to have a better numerical behavior than the direct Gram-Schmidt
orthogonalization process {\cite{trefethen_numerical_1997}}[Lecture
8].  It computes the polynomial
$\tmop{proj} (f, \tmmathbf{p}, \tmmathbf{m})$ characterized by the
following lemma.

\begin{lemma}
  \label{lem:proj}If $\langle p_i, m_j \rangle_{\sigma} = 0$ if $j < i$ \ and
  $\langle p_i, m_i \rangle_{\sigma} = 1$, 
there is a unique polynomial $g$ such that $g = f - \sum_{i = 1}^l
  \lambda_i p_i$ with $\lambda_i \in \mathbbm{K}$ and $\langle g, m_i
  \rangle_{\sigma} = 0$ for $i = 1, \ldots, l$. 
\end{lemma}
\begin{proof}
  We prove by induction on the index $i$ of the loop that $g$ is orthogonal to
  $[m_1, \ldots, m_i]$. For $i = 1$, $g = f - \langle f, m_1 \rangle_{\sigma}
  p_1$ is such that $\langle g, m_1 \rangle_{\sigma} = \langle f, m_1
  \rangle_{\sigma} - \langle f, m_1 \rangle_{\sigma}  \langle p_1, m_1
  \rangle_{\sigma}$ $ = 0$.

  If the property is true at step $k \leqslant l$, i.e.
  $\langle g, m_i \rangle_{\sigma} = 0$ for $i < k$, then $g' = g - \langle g,
  m_k \rangle_{\sigma} p_k$ is such that $\langle g, m_i \rangle_{\sigma} -
  \langle g, m_k \rangle_{\sigma} \langle p_k, m_i \rangle_{\sigma} = \langle
  g, m_i \rangle_{\sigma} = 0$ by induction hypothesis. By construction,
  $\langle g', m_k \rangle = \langle g, m_k \rangle_{\sigma} - \langle g, m_k
  \rangle_{\sigma} \langle p_k, m_k \rangle_{\sigma} = 0$ and the induction
  hypothesis is true for $k$. As the matrix $(\langle p_{\tmop{ij}}, m_i
  \rangle_{\sigma})_{1 \leqslant i, j \leqslant l}$ is invertible, there exists a
  unique polynomial of the form $g = f - \sum_{j = 1}^l \lambda_j p_j$, such
  that $\langle g, m_i \rangle_{\sigma} = 0$ for $i = 1, \ldots, l$, which
  concludes the proof of the lemma.
\end{proof}

Algorithm \ref{algo:1} for computing a border basis of $\mathcal{A}_{\sigma}$
proceeds inductively, starting from $\tmmathbf{p}= [], \tmmathbf{m}= [],
\tmmathbf{b}= []$, extending the basis $\tmmathbf{p}$ with a new polynomial
$p_{\alpha}$, orthogonal to the vector space spanned by
$\tmmathbf{m}$ for the inner product $\langle \cdummy, \cdummy
\rangle_{\sigma}$, extending $\tmmathbf{m}$ with a new monomial $m_{\alpha}$,
such that $\langle p_{\alpha}, m_{\alpha} \rangle_{\sigma} = 1$ and $\langle
p_{\beta}, m_{\alpha} \rangle = 0$ for $\beta \in \tmmathbf{b}$ and extending
$\tmmathbf{b}$ with $\alpha$.

{\begin{algorithm}[ht]\caption{\label{algo:1}Artinian Gorenstein border basis }{\tmstrong{Input:} the coefficients
$\sigma_{\alpha}$ of a series $\sigma \in \mathbbm{K} [[\tmmathbf{y}]]$ for
$\alpha \in \tmmathbf{a} \subset \mathbbm{N}^n$ with $\tmmathbf{a}$ a finite
set of exponents connected to $\mathbf{0}$.
\begin{itemizeminus}
  \item Let $\tmmathbf{b} \assign []$; $\tmmathbf{c} \assign []$;
  $\tmmathbf{d}= []$; $\tmmathbf{n} \assign [\tmmathbf{0}]$; $\tmmathbf{s}
  \assign \tmmathbf{a}$; $\tmmathbf{t} \assign \tmmathbf{a}$;
  
  \item while $\tmmathbf{n} \neq \emptyset$ do
  \begin{itemizeminus}
    \item for each $\alpha \in \tmmathbf{n}$,
    \begin{enumeratealpha}
      \item $p_{\alpha} \assign \tmop{proj} (\tmmathbf{x}^{\alpha},
      [p_{\beta}]_{\beta \in \tmmathbf{b}}, [m_{\beta}]_{\beta \in
      \tmmathbf{b}})$;
      
      \item find the first $\gamma \in \tmmathbf{t}$ such that
      $\tmmathbf{x}^{\gamma} p_{\alpha} \in \langle
      \tmmathbf{x}^{\tmmathbf{a}} \rangle$ and $\langle p_{\alpha},
      \tmmathbf{x}^{\gamma} \rangle_{\sigma} \neq 0$;
      
      \item if such a $\gamma$ exists then
      
      \ \ let $m_{\alpha} \assign \frac{1}{\langle p_{\alpha},
      \tmmathbf{x}^{\gamma} \rangle_{\sigma}} \tmmathbf{x}^{\gamma}
      ;$
      
      \hspace{-2cm} {\small\tmverbatim{[optional]}} \ \ \ \ \,let $q_{\alpha} \assign \tmop{proj} (m_{\alpha}, [q_{\beta}]_{\beta
      \in \tmmathbf{b}}, [p_{\beta}]_{\beta \in \tmmathbf{b}})$;
      
      \ \ add $\alpha$ to $\tmmathbf{b}$; remove $\alpha$ from
      $\tmmathbf{s}$;
      
      \ \ add $\gamma$ to $\tmmathbf{c}$; remove $\gamma$ from
      $\tmmathbf{t}$;
      
      else
      
      \ \ let $k_{\alpha} = p_{\alpha}$;
      
      \ \ add {\tmem{}}$\alpha$ to $\tmmathbf{d}$; remove $\alpha$ from
      $\tmmathbf{s}$;
      
      end;
    \end{enumeratealpha}
  \item $\tmmathbf{n} \assign \tmop{next} (\tmmathbf{b}, \tmmathbf{d},
    \tmmathbf{c}, \tmmathbf{s}) ;$
  \end{itemizeminus}
\end{itemizeminus}
{\tmstrong{Output:}}
\begin{itemizeminus}
  \item exponent sets $\tmmathbf{b}= [\beta_1, \ldots, \beta_r]$, $\tmmathbf{c}= [\gamma_1, \ldots, \gamma_r]$.
  
  \item bases $\tmmathbf{p}= [p_{\beta_i}]$, %$\tmmathbf{m}= [m_{\beta_i}]$,
  {\small\tmverbatim{[optional]}} $\tmmathbf{q}= [q_{\beta_i}]$.
  
  \item the relations $\tmmathbf{k}= [p_{\alpha}]_{\alpha \in \tmmathbf{d}}$
  where $p_{\alpha} =\tmmathbf{x}^{\alpha} - \sum_{i = 1}^{r} \lambda_{\beta_i}
  p_{\beta_i}$ for $\alpha \in \tmmathbf{d}$.
\end{itemizeminus}}
\end{algorithm}} The main difference with Algorithm 4.1 in
{\cite{mourrain_polynomial-exponential_2016}} is the projection
procedure and the list of monomials $\tmmathbf{s}$, $\tmmathbf{t}$
used to generate new monomials and to perform the projections. The
lists
$\tmmathbf{b}, \tmmathbf{d}, \tmmathbf{c}, \tmmathbf{s}, \tmmathbf{t}$
are lists of exponents, identified with monomials.

We verify that at each loop of the algorithm, the lists 
$\tmmathbf{b}$, $\tmmathbf{d}$ and $\tmmathbf{s}$ 
are disjoint and 
$\tmmathbf{b} \cup \tmmathbf{d} \cup \tmmathbf{s}=\tmmathbf{a}$. 
We also verify that $m_{\alpha}$ are monomials up 
to a scalar, that the set of their exponents is $\tmmathbf{c}$, that 
$\tmmathbf{c}$ and $\tmmathbf{t}$ are disjoint and that $\tmmathbf{c} \cup 
\tmmathbf{t}=\tmmathbf{a}$. 

The algorithm uses the function $\tmop{next} (\tmmathbf{b}, \tmmathbf{d},
\mathbf{c}, \tmmathbf{s})$, which computes the set of monomials $\nw$ in $\partial \tmmathbf{b}
\cap \tmmathbf{s}$, which are not in $\tmmathbf{d}$ and such $\nw\cdummy \mathbf{c}\subset \mathbf{a}
=\tmmathbf{b} \cup \tmmathbf{d} \cup \tmmathbf{s}$. 

We denote by $\prec$ the order induced by the treatment of the monomials of
$\tmmathbf{a}$ in the loops of the algorithm, so that the monomials treated at
the $l^{\tmop{th}}$ loop are smaller than the monomials in $\tmmathbf{n}$ at
the $(l + 1)^{\tmop{th}} $ loop. For $\alpha \in \tmmathbf{a}$, we denote by
$\tmmathbf{b}_{\prec \alpha}$ the list of monomial exponents $\beta \in
\tmmathbf{b}$ with $\beta \prec \alpha$ and by $B_{\prec \alpha}$ the vector
space spanned by these monomials. For $\alpha \in \tmmathbf{b}$, let
$\tmmathbf{b}_{\preccurlyeq \alpha} =\tmmathbf{b}_{\prec \alpha} \cup
[\alpha]$.

The following properties are also satisfied during this algorithm:
\begin{lemma}
  \label{lem:semiortho}For $\alpha \in \tmmathbf{b}$, we have $\forall \beta
  \in \tmmathbf{b}_{\prec \alpha}$, $\langle p_{\alpha}, m_{\beta}
  \rangle_{\sigma} = 0$ and $\langle p_{\alpha}, m_{\alpha} \rangle_{\sigma} =
  1$. For $\alpha \in \tmmathbf{d}$, $\langle p_{\alpha},
  \tmmathbf{x}^{\gamma} \rangle_{\sigma} = 0$ for all $\gamma \in
  \tmmathbf{a}$ such that $\tmmathbf{x}^{\gamma} p_{\alpha} \in
  \langle \tmmathbf{x}^{\tmmathbf{a}} \rangle$.
\end{lemma}
\begin{proof}
  By construction, 
  \[ p_{\alpha} = \tmop{proj} (\tmmathbf{x}^{\alpha}, [p_{_{\beta}}]_{\beta
     \in \tmmathbf{b}_{\prec \alpha}}, [m_{\beta}]_{\beta \in
     \tmmathbf{b}_{\prec \alpha}}) \]
  is orthogonal to $m_{\beta}$ for $\beta \in \tmmathbf{b}_{\prec \alpha}$. We
  consider two exclusive cases: $\alpha \in \tmmathbf{b}$ and $\alpha \in
  \tmmathbf{d}$.
  \begin{itemize}
    \item If $\alpha \in \tmmathbf{b}$, then there exists
    $\tmmathbf{x}^{\gamma} \in \tmmathbf{s}$ such that $\langle p_{\alpha},
    \tmmathbf{x}^{\gamma} \rangle_{\sigma} \neq 0$. Thus $m_{\alpha} =
    \frac{1}{\langle p_{\alpha}, \tmmathbf{x}^{\gamma} \rangle_{\sigma}}
    \tmmathbf{x}^{\gamma}$ is such that $\langle p_{\alpha}, m_{\alpha}
    \rangle_{\sigma} = 1$. By construction, $\langle p_{\alpha}, m_{\beta}
    \rangle_{\sigma} = 0$ for $\beta \in \tmmathbf{b}_{\prec \alpha}$. \
    
    \item If $\alpha \in \tmmathbf{d}$, then there is no
    $\tmmathbf{x}^{\gamma} \in \tmmathbf{s}$ such that $\langle p_{\alpha},
    \tmmathbf{x}^{\gamma} \rangle_{\sigma} \neq 0$ and $\tmmathbf{x}^{\gamma}
    p_{\alpha} \in \langle \tmmathbf{x}^{\tmmathbf{a}} \rangle$. Thus
    $p_{\alpha}$ is orthogonal to $\tmmathbf{x}^{\gamma}$ for all $\gamma \in
    \tmmathbf{s}$ with $\tmmathbf{x}^{\gamma} p_{\alpha} \in \langle
    \tmmathbf{x}^{\tmmathbf{a}} \rangle$. By construction, $p_{\alpha}$ is
    orthogonal to $m_{\beta}$ for $\beta \in \tmmathbf{b}$. As $\tmmathbf{b}
    \cup \tmmathbf{s}=\tmmathbf{a}$, \ $\langle p_{\alpha},
    \tmmathbf{x}^{\gamma} \rangle_{\sigma} = 0$ for all $\gamma \in
    \tmmathbf{a}$ such that $\tmmathbf{x}^{\gamma} p_{\alpha} \in
    \langle \tmmathbf{x}^{\tmmathbf{a}} \rangle$.
  \end{itemize}
  This concludes the proof of this lemma.
\end{proof}

\begin{lemma}
  \label{lem:ortho}For $\alpha \in \tmmathbf{b}$, $\langle m_{\beta}
  \rangle_{\beta \in \tmmathbf{b}_{\preccurlyeq \alpha}} = \langle q_{\beta}
  \rangle_{\beta \in \tmmathbf{b}_{\preccurlyeq \alpha}}$ and the bases
  $\tmmathbf{p}= [p_{\beta}]_{\beta \in \tmmathbf{b}_{\preccurlyeq \alpha}},
  \tmmathbf{q}= [q_{\beta}]_{\beta \in \tmmathbf{b}_{\preccurlyeq \alpha}}$
  are pairwise orthogonal.
\end{lemma}

\begin{proof}
  We prove it by induction on $\alpha$. If $\alpha =\tmmathbf{0}$ is not in
  {\tmem{{\tmstrong{b}}}}, then $\sigma_{\alpha} = 0$ for all $\alpha \in
  \tmmathbf{a}$, $\tmmathbf{p}$ and $\tmmathbf{q}$ are empty and the property
  is satisfied. If $\alpha =\tmmathbf{0}$ is in $\tmmathbf{b}$, then
  $p_{\alpha} = 1$ and $q_{\alpha} = m_{\alpha}$ is such that $\langle
  p_{\alpha}, m_{\alpha} \rangle_{\sigma} = 1$. The property is true for
  $\alpha =\tmmathbf{0}$.
  
  Suppose that it is true for all $\beta \in \tmmathbf{b}_{\prec \alpha}$. By
  construction, the polynomial $q_{\alpha} = \tmop{proj} (m_{\alpha},
  [q_{_{\beta}}]_{\beta \in \tmmathbf{b}_{\prec \alpha}},$ $[p_{\beta}]_{\beta
  \in \tmmathbf{b}_{\prec \alpha}})$ is orthogonal to $p_{\beta}$ for $\beta
  \prec \alpha$. By induction hypothesis, $[p_{\beta}]_{\beta \in
  \tmmathbf{b}_{\prec \alpha}}, \tmmathbf{q}= [q_{\beta}]_{\beta \in
  \tmmathbf{b}_{\prec \alpha}}$ are pairwise orthogonal, thus
  \[ q_{\alpha} = m_{\alpha} - \sum_{\beta \in \tmmathbf{b}_{\prec \alpha}}
     \langle p_{\beta}, m_{\alpha} \rangle_{\sigma} q_{\beta} . \]
  By the induction hypothesis, we deduce that
  \begin{eqnarray*}
    \langle m_{\beta} \rangle_{\beta \in \tmmathbf{b}_{\preccurlyeq \alpha}} &
    = & \langle m_{\beta} \rangle_{\beta \in \tmmathbf{b}_{\prec \alpha}} +
    \langle m_{\alpha} \rangle = \langle q_{\beta} \rangle_{\beta \in
    \tmmathbf{b}_{\prec \alpha}} + \langle m_{\alpha} \rangle\\
    & = & \langle q_{\beta} \rangle_{\beta \in \tmmathbf{b}_{\prec \alpha}} +
    \langle q_{\alpha} \rangle = \langle q_{\beta}
    \rangle_{\tmmathbf{b}_{\preccurlyeq \alpha}} .
  \end{eqnarray*}
  By Lemma \ref{lem:semiortho}, $p_{\alpha}$ is orthogonal to $m_{\beta}$ for
  $\beta \in \tmmathbf{b}_{\prec \alpha}$ and thus to $q_{\beta}$ for $\beta
  \in \tmmathbf{b}_{\prec \alpha}$. We deduce that
\begin{eqnarray*}
 \langle p_{\alpha}, q_{\alpha} \rangle_{\sigma} &=& \langle p_{\alpha},
     m_{\alpha} \rangle_{\sigma} - \sum_{\beta \in \tmmathbf{b}_{\prec
     \alpha}} \langle p_{\beta}, m_{\alpha} \rangle_{\sigma} \langle
     p_{\alpha}, q_{\beta} \rangle_{\sigma} \\&=& \langle p_{\alpha}, m_{\alpha}
     \rangle_{\sigma} = 1. 
\end{eqnarray*}
This shows that $[p_{\beta}]_{\beta \in \tmmathbf{b}_{\preccurlyeq \alpha}}$ 
  and $\tmmathbf{q}= [q_{\beta}]_{\beta \in \tmmathbf{b}_{\preccurlyeq
  \alpha}}$ are pairwise orthogonal and concludes the proof by induction.
\end{proof}

\begin{lemma}
  \label{lem:leadingterm}At the $l^{\tmop{th}}$ loop of the algorithm, the
  polynomials $p_{\alpha}$ for $\alpha \in
  \tmmathbf{n}$ are of the form
  $p_{\alpha} =\tmmathbf{x}^{\alpha} + b_{\alpha}$ with
  $b_{\alpha} \in B_{\prec \alpha}$.
\end{lemma}

\begin{proof}
  We prove by induction on the loop index $l$ that we have \ $p_{\alpha}
  =\tmmathbf{x}^{\alpha} + b_{\alpha}$ with $b_{\alpha} \in B_{\prec \alpha}$.
  
  The property is clearly true for $l = 0$, $\alpha =\tmmathbf{0}$ and
  $p_{\alpha} = 1 =\tmmathbf{x}^{\tmmathbf{0}}$. Suppose that it is true for
  any $l' < l$ and consider the $l^{\tmop{th}}$ loop of the algorithm. The
  polynomial $p_{\alpha}$ is constructed by projection of
  $\tmmathbf{x}^{\alpha}$ on $\langle p_{\alpha} \rangle_{\beta \in
  \tmmathbf{b}}$ orthogonally to $\langle m_{\beta} \rangle_{\beta \in
  \tmmathbf{b}}$ where $\tmmathbf{b}=\tmmathbf{b}_{\prec \alpha}$. By
  induction hypothesis, $p_{\beta} =\tmmathbf{x}^{\beta} + b_{\beta}$ with
  $b_{\beta} \in B_{\prec \beta} \subset B_{\prec \alpha}$. Then by Lemma
  \ref{lem:proj}, we have
  \[ p_{\alpha} =\tmmathbf{x}^{\alpha} + \sum_{\beta \prec \alpha}
     \lambda_{\beta} p_{\beta} \noplus =\tmmathbf{x}^{\alpha} \noplus +
     b_{\alpha} \]
  with $\lambda_{\beta} \in \mathbbm{K}$, $b_{\alpha} \in B_{\prec \alpha}$.
  Thus, the induction hypothesis is true for $l$, which concludes
  the proof.
\end{proof}

\subsection{Quotient algebra structure}

We show now that the algorithm outputs a border basis of an Artinian
Gorenstein algebra $\mathcal{A}_{\tilde{\sigma}}$ for an extension
$\tilde{\sigma}$ of $\sigma$, when all the border relations are
computed, that is, when $\tmmathbf{d}= \partial \tmmathbf{b}$.

\begin{theorem}
  \label{thm:flatextalgo}Let $\tmmathbf{b}= [\beta_1, \ldots, \beta_r]$,
  $\tmmathbf{c}= [\gamma_1, \ldots, \gamma_r]$, $\tmmathbf{p}= [p_{\beta_1},
  \ldots, p_{\beta_r}]$, $\tmmathbf{q}= [q_{\beta_1}, \ldots, q_{\beta_r}]$
  and $\tmmathbf{k}= [p_{\alpha_1}, \ldots, p_{\alpha_s}]$ be the output of
  Algorithm \ref{algo:1}. Let $V = \langle \tmmathbf{x}^{\tmmathbf{b}^+}
  \rangle$. If $\tmmathbf{d}= \partial \tmmathbf{b}$ and $\tmmathbf{c}^+
  \subset \tmmathbf{b}'$ connected to $1$ such that
  $\tmmathbf{x}^{\tmmathbf{b}^+} \cdot \tmmathbf{x}^{\tmmathbf{b}'}
  =\tmmathbf{x}^{\tmmathbf{a}}$ then $\sigma$ coincides on $\langle
  \tmmathbf{x}^{\tmmathbf{a}} \rangle$ with a series $\tilde{\sigma} \in
  \mathbbm{K} [[\tmmathbf{y}]]$ such that
  \begin{itemize}
    \item $\tmop{rank} H_{\tilde{\sigma}} = r$,
    
    \item $(\tmmathbf{p}, \tmmathbf{q})$ are pairwise orthogonal bases of
    $\mathcal{A}_{\tilde{\sigma}}$ for the inner product $\langle \cdummy,
    \cdummy \rangle_{\tilde{\sigma}}$,
    
    \item The family $\tmmathbf{k}= \{ p_{\alpha}, \alpha \in \partial
    \tmmathbf{b} \}$ is a border basis of the ideal $I_{\tilde{\sigma}}$, with
    respect to $\tmmathbf{x}^{\tmmathbf{b}}$.
    
    \item The matrix of multiplication by $x_k$ in the basis $\tmmathbf{p}$
    (resp. {\tmem{{\tmstrong{q}}}}) of $\mathcal{A}_{\tilde{\sigma}}$ is $M_k
    \assign (\langle \sigma | x_k p_{\beta_j} q_{\beta_i} \rangle)_{1
    \leqslant i, j \leqslant r} \nobracket$ (resp. $M_k^t$).
\end{itemize}
\end{theorem}

\begin{proof}
  By construction, $\tmmathbf{x}^{\tmmathbf{b}^+}$ is connected to $1$. Let $V
  = \langle \tmmathbf{x}^{\tmmathbf{b}^+} \rangle$ and $V' = \langle
  \tmmathbf{x}^{\tmmathbf{b}'} \rangle$. As $\tmmathbf{b}^+ =\tmmathbf{b} \cup
  \tmmathbf{d}$, a basis of $V$ is formed by the monomials
  $\tmmathbf{x}^{\tmmathbf{b}}$ and the polynomials $p_{\alpha}
  =\tmmathbf{x}^{\alpha} + b_{\alpha}$ with $b_{\alpha} \in \langle
  \tmmathbf{x}^{\tmmathbf{b}} \rangle$ for $\alpha \in \tmmathbf{d}$. The
  matrix of $H_{\sigma}^{V, V'}$ in this basis of $V$ and a basis of
  $V',$ which first elements are $m_{\beta_1}, \ldots,
  m_{\beta_r}$, is of the form 
  \[ H_{\sigma}^{V, V'} = \left(\begin{array}{cc}
       L_r & 0\\
       \ast & 0
     \end{array}\right) \]
  where $L_r$ is a lower triangular invertible matrix of size $r$. The kernel
  of $H_{\sigma}^{V, V'}$ is generated by the polynomials $p_{\alpha}$ for
  $\alpha \in \tmmathbf{d}$.
  
  By Theorem \ref{thm:flatext}, $\sigma$ coincides on $V \cdummy V' =
  \langle \tmmathbf{x}^{\tmmathbf{a}} \rangle$ with a series $\tilde{\sigma}$
  such that $\tmmathbf{x}^{\tmmathbf{b}}$ is a basis of
  $\mathcal{A}_{\bar{\sigma}} =\mathbbm{K} [\tmmathbf{x}] /
  I_{\tilde{\sigma}}$ and $I_{\tilde{\sigma}} = (\ker H_{\tilde{\sigma}}^{V,
  V'}) = (p_{\alpha})_{\alpha \in \tmmathbf{d}}$.
  
  By Lemma \ref{lem:leadingterm}, $p_{\alpha} =\tmmathbf{x}^{\alpha} +
  b_{\alpha}$ with $\alpha \in \partial \tmmathbf{b}$ and $b_{\alpha} \in
  \langle \tmmathbf{x}^{\tmmathbf{b}} \rangle$. Thus $(p_{\alpha})_{\alpha \in
  \partial \tmmathbf{b}}$ is a border basis with respect to
  $\tmmathbf{x}^{\tmmathbf{b}}$ for the ideal $I_{\tilde{\sigma}}$, since 
  $\tmmathbf{x}^{\tmmathbf{b}}$ is a basis of of $\mathcal{A}_{\bar{\sigma}}$.
  This shows that $\tmop{rank} H_{\tilde{\sigma}} = \dim
  \mathcal{A}_{\tilde{\sigma}} = | \tmmathbf{b} | = r$.
  
  By Lemma \ref{lem:ortho}, $(\tmmathbf{p}, \tmmathbf{q})$ are pairwise
  orthogonal for the inner product $\langle \cdummy, \cdummy
  \rangle_{\sigma}$, which coincides with $\langle \cdummy, \cdummy
  \rangle_{\tilde{\sigma}}$ on $\langle \tmmathbf{x}^{\tmmathbf{a}} \rangle$.
  Thus they are pairwise orthogonal bases of $\mathcal{A}_{\tilde{\sigma}}$
  for the inner product $\langle \cdummy, \cdummy \rangle_{\tilde{\sigma}}$.
  
  As we have $x_k p_{\beta_j} \equiv \sum_{i = 1}^{r} \langle x_k
  p_{\beta_j} \nobracket, q_{\beta_i} \rangle_{\sigma} \nobracket
  p_{\beta_i}$, the matrix of multiplication by $x_k$ in the basis
  $\tmmathbf{p}$ of $\mathcal{A}_{\tilde{\sigma}}$ is $$M_k \assign (\langle
  x_k p_{\beta_j} \nobracket, q_{\beta_i} \rangle_{\sigma} \nobracket)_{1
  \leqslant i, j \leqslant r} = (\langle \sigma | x_k p_{\beta_j} q_{\beta_i}
  \rangle)_{1 \leqslant i, j \leqslant r}. $$ Exchanging the role of
  {\tmstrong{p}} and {\tmstrong{q}}, we obtain $M_k^t$ for the matrix of
  multiplication by $x_k$ in the basis {\tmem{{\tmstrong{q}}}}.\tmtextbf{}
\end{proof}

\begin{lemma}
  If $\prec$ is a monomial ordering and if at the end of the algorithm
  $\tmmathbf{d}= \partial \tmmathbf{b}$ and $\tmmathbf{c}^+ \subset
  \tmmathbf{b}'$ connected to $1$ with $\tmmathbf{x}^{\tmmathbf{b}^+} \cdot
  \tmmathbf{x}^{\tmmathbf{b}'} =\tmmathbf{x}^{\tmmathbf{a}}$, then
  $\tmmathbf{b}=\tmmathbf{c}$ and $\tmmathbf{k}$ is a Gr{\"o}bner basis of the
  ideal $I_{\sigma}$ for the monomial ordering.
\end{lemma}

\begin{proof}
  If $\prec$ is a monomial ordering, then the polynomials $p_{\alpha}
  =\tmmathbf{x}^{\alpha} + b_{\alpha}$, $\alpha \in \partial \tmmathbf{b}$ are
  constructed in such a way that their leading term is
  $\tmmathbf{x}^{\alpha}$. Therefore the border basis $\tmmathbf{k}=
  (p_{\alpha})_{\alpha \in \partial \tmmathbf{b}}$ is also a Gr{\"o}bner
  basis.
  
  By construction, $\tmmathbf{c}$ is the set of monomials $\gamma \in
  \tmmathbf{a}$ such that $\langle p_{\beta}, \tmmathbf{x}^{\gamma}
  \rangle_{\sigma} \neq 0$ for some $\beta \in \tmmathbf{b}$. Suppose that
  $\gamma \in \tmmathbf{c}$ is not in $\tmmathbf{b}$. Then
  $\tmmathbf{x}^{\gamma} \in (\tmmathbf{x}^{\tmmathbf{d}})$ and there is
  $\delta \in \tmmathbf{d}$ and $\gamma' \in \tmmathbf{a}$ such that $\gamma =
  \delta + \gamma'$. As $p_{\delta} \in \tmmathbf{k}$, we have $\langle
  p_{\delta}, \tmmathbf{x}^{\alpha} \rangle_{\sigma} = 0$ for $\alpha \in
  \tmmathbf{a}$ such that $p_{\delta} \tmmathbf{x}^{\alpha} \in \langle
  \tmmathbf{x}^{\tmmathbf{a}} \rangle$. By Lemma \ref{lem:leadingterm},
  $p_{\delta} =\tmmathbf{x}^{\delta} + b_{\delta}$ with
  $b_{\delta} \in \tmmathbf{b}_{\prec \delta}$ with $\tmmathbf{x}^{\delta}
  \succ b_{\delta}$. \
  \[ \langle p_{\beta}, \tmmathbf{x}^{\gamma} \rangle_{\sigma} = \langle
     p_{\beta}, \tmmathbf{x}^{\delta} \tmmathbf{x}^{\gamma'} \rangle_{\sigma}
     = \langle p_{\beta}, p_{\delta} \tmmathbf{x}^{\gamma'} \rangle_{\sigma}
     \noplus \noplus - \langle p_{\beta}, b_{\delta} \tmmathbf{x}^{\gamma'}
     \rangle_{\sigma} . \]
  We have $\langle p_{\beta}, p_{\delta} \tmmathbf{x}^{\gamma'}
  \rangle_{\sigma} \noplus \noplus = \langle p_{\delta}, p_{\beta}
  \tmmathbf{x}^{\gamma'} \rangle_{\sigma} \noplus \noplus = 0$ since
  $p_{\delta} \in \tmmathbf{k}$ and $p_{\delta} p_{\beta}
  \tmmathbf{x}^{\gamma'} \in \langle \tmmathbf{x}^{\tmmathbf{a}} \rangle$. As
  $\gamma$ is the first monomial of $\tmmathbf{a}$ such that $\langle
  p_{\beta}, \tmmathbf{x}^{\gamma} \rangle_{\sigma} \neq 0$ and $b_{\delta}
  \tmmathbf{x}^{\gamma'} \prec \tmmathbf{x}^{\delta + \gamma'}
  =\tmmathbf{x}^{\gamma}$, we have $\langle p_{\beta}, b_{\delta}
  \tmmathbf{x}^{\gamma'} \rangle_{\sigma}$, which implies that $\langle
  p_{\beta}, \tmmathbf{x}^{\gamma} \rangle_{\sigma} = 0$. This is in
  contradiction with the hypothesis $\langle p_{\beta}, \tmmathbf{x}^{\gamma}
  \rangle_{\sigma} \neq 0$, therefore $\gamma \in \tmmathbf{b}$. We deduce
  that \tmtextbf{}$\tmmathbf{c} \subset \tmmathbf{b}$ and the equality holds
  since the two sets have the same cardinality.
\end{proof}

Notice that to construct a minimal reduced Gr{\"o}bner basis of
$I_{\tilde{\sigma}}$ for the monomial ordering $\prec$, it suffices to
keep the elements $p_{\alpha} \in \tmmathbf{k}$ with $\alpha$ minimal
for the component-wise partial ordering.

\subsection{Complexity}

Let $s = | \tmmathbf{a} |$ and $r = | \tmmathbf{b} |$, $\delta = | \partial
\tmmathbf{b} |$. As $\tmmathbf{b} \subset \tmmathbf{a}$ and the monomials in
$\partial \tmmathbf{b}$ are the product of a monomial in $\tmmathbf{b}$ by one
of the variables $x_1, \ldots, x_n$, we have $r \leqslant s$ and $\delta
\leqslant n r$.

\begin{proposition}
  \label{prop:complexity}The complexity of the algorithm to compute the bases
  $\tmmathbf{p}$ and $\tmmathbf{q}$ is $\mathcal{O} ((r + \delta) r s)$.
\end{proposition}

\begin{proof}
  At each step, the computation of $p_{\alpha}$ (resp. $q_{\alpha}$) requires
  $\mathcal{O} (r^2)$ arithmetic operations, since the support of the
  polynomials $p_{\beta}$, $q_{\beta}$ $(\beta \in \tmmathbf{b})$ is in
  $\tmmathbf{b}$ and $| \tmmathbf{b} | \leqslant r$. Computing $\langle
  \tmmathbf{x}^{\gamma}, p_{\alpha} \rangle_{\sigma}$ for all $\gamma \in
  \tmmathbf{t}$ requires $\mathcal{O} (r s)$ arithmetic operations. As the
  number of polynomials $p_{\alpha}$ is at most $| \tmmathbf{b}^+ | = r +
  \delta$, the total cost for computing $\tmmathbf{p}$ and $\tmmathbf{q}$ is
  thus in $\mathcal{O} ((r + \delta) (r^2 + r s)) =\mathcal{O} ((r + \delta)
  \noplus r s)$.
\end{proof}

As $\delta \leqslant n r$, the complexity of this algorithm is in $\mathcal{O}
(n r^2 s)$.

The algorithm is connected to the {\tmem{Berlekamp-Massey-\-Sakata}} algorithm,
which computes a Gr{\"o}bner basis for a monomial ordering $\prec$. In the BMS
algorithm, a minimal set $\mathcal{F}$ of recurrence polynomials valid for the
monomials smaller that a given monomial $m$ is computed. A monomial basis
$\tmmathbf{b}^{\ast}$ generated by all the divisors of some corner elements is
constructed. The successor $m^+$ of the monomial $m$ for the monomial ordering
$\prec$ is considered and the family $\mathcal{F}$ of valid recurrence
polynomials is updated by computing their discrepancy at the monomial $m^+$
and by cancelling this discrepancy, if necessary, by combination with one
lower polynomial {\cite{saints_algebraic-geometric_1995}}.

Let $\delta$ be the size of the border $\partial \tmmathbf{b}^{\ast}$ of the
monomial basis $\tmmathbf{b}^{\ast}$computed by BMS algorithm. At each update,
there are at most $\delta$ polynomials in $\mathcal{F}$. Let $s'$ be the
maximum number of their non-zero terms. Then the update of $\mathcal{F}$
requires $\mathcal{O} (\delta s')$ arithmetic operations. The number of
updates is bounded by the number $r + \delta$ of monomials in
$\tmmathbf{b}^+$. Checking the discrepency of a polynomial in $\mathcal{F}$
for all the monomials in $\tmmathbf{x}^{\tmmathbf{a}}$ requires $\mathcal{O}
(s' s)$ arithmetic operations. Thus, the total cost of the BMS algorithm is in
$\mathcal{O} ((r + \delta) \delta s' + \delta s' s)$. As the output
polynomials in the Gr{\"o}bner basis are not necessarily reduced, the maximal
number of terms $s' \leqslant s$ can be of the same order than $s$. Thus the
complete complexity of BMS algorithm is in $\mathcal{O} (\delta s^2)
=\mathcal{O} (n r s^2)$, which is an order of magnitude larger than
the bound of Proposition \ref{prop:complexity}, assuming that $r\ll s$.

The method presented in {\cite{berthomieu_linear_2015}} for computing
a Gr{\"o}bner basis of the recurrence polynomials computes the rank of
a Hankel matrix of size the number $\tilde{s}$ of monomials of degree
$\leqslant d$ for a bound $d$ on the degree of the recurrence
relations. It deduces a monomial basis $\tmmathbf{b}$ stable by division
and obtains the valid recurrence relations for the border monomials by
solving a linear Hankel system of size $r$. Thus the complexity is in
$\mathcal{O} (\delta r^{\omega} + \tilde{s}^{\omega})$ where
$2.3 \leqslant \omega \leqslant 3$. It is also larger than the bound
of Proposition \ref{prop:complexity}.  This bound could be improved by
exploiting the rank displacement of the structured matrices involved
in this method \cite{bostan_solving_2008}, but the known bounds on the
displacement rank of the matrices involved in the computation do not
improve the bound of Proposition \ref{prop:complexity}.

\section{Examples \nopunct}

\subsection{Multivariate Prony method}

Given a function $h (u_1, \ldots, u_n) = \sum_{i = 1}^r \omega_i e^{\zeta_{i,
1} u_1 + \cdots + \zeta_{i, n} u_n}$, the problem is to compute its
decomposition as a weighted sum of exponentials, from values of $h$. The
method proposed by G. Riche de Prony for sums of univariate exponential
functions consists in sampling the function at regularly spaced values
{\cite{baron_de_prony_essai_1795}}. In the multivariate extension of this
method, the function is sampled on a grid in $\mathbbm{R}^n$, for
instance $\mathbbm{N}^n$. The decomposition is computed from a subset of the
multi-index sequence of evaluation $\sigma_{\alpha} = h (\alpha_1, \ldots,
\alpha_n)$ for $\alpha = (\alpha_1, \ldots, \alpha_n) \in \mathbbm{N}^n$. The
ideal $I_{\sigma}$ associated to this sequence is the ideal defining the
points $\xi_i = (e^{\zeta_{i, 1}}, \ldots, e^{\zeta_{i, 1 n}})$. To compute
this decomposition, we apply the border basis algorithm to the sequence
$\sigma_{\alpha}$ for $| \alpha | \leqslant d$ with $d$ high enough, and
obtain a border basis of the ideal $I_{\sigma}$ defining the points $\xi_1,
\ldots, \xi_r \in \mathbbm{K}^n$, a basis of $\mathcal{A}_{\sigma}$ and the
tables of multiplication in this basis. By applying the decomposition
algorithm in {\cite{mourrain_polynomial-exponential_2016}}, we deduce the
points $\xi_i = (e^{\zeta_{i, 1}}, \ldots, e^{\zeta_{i, 1 n}})$. Taking the
log of their coordinates  $\log(\xi_{i,j})=\zeta_{i,j}$ yields the coordinates of the frequencies $\zeta_i$.

\subsection{Fast decoding of algebraic-geometric codes}

Let $\mathbbm{K}$ be a finite field. We consider an algebraic-geometric code
$C$ obtained by evaluation of polynomials in $\mathbbm{K} [x_1, \ldots, x_n]
\overset{}{}$ of degree $\leqslant d$ at points $\xi_1, \ldots, \xi_l \in
\mathbbm{K}^n$. It is a finite vector space in $\mathbbm{K}^l$. We use the
words of the orthogonal code $C^{\perp} = \{ (m_1, \ldots, m_l) \mid m
\nosymbol \cdot c = m_1 c_1 + \cdots + m_l c_l = 0 \}$ for the transmission of
information. Suppose that an error $\omega = (\omega_1, \ldots, \omega_l)$
occurs in the transmission of a message $m = (m_1, \ldots, m_l)$ so that the
message $m^{\ast} = m + \omega$ is received. Let $\omega_{i_1}, \ldots,
\omega_{i_r}$ be the non-zero coefficients of the error vector $\omega$. To
correct the message $r$, we use the moments or syndromes $\sigma_{\alpha} =
(\xi_1^{\alpha}, \ldots, \xi_l^{\alpha}) \cdot m^{\ast} = (\xi_1^{\alpha},
\ldots, \xi_l^{\alpha}) \cdot \omega = \sum_{j = 1}^r w_{i_j}
\xi_{i_j}^{\alpha}$ for $\alpha = (\alpha_1, \ldots, \alpha_n) \in
\mathbbm{N}^n$ with $| \alpha | \leqslant d$. We compute generators of the set
of error-locator polynomials, that is, the polynomials vanishing at the points
$\xi_{i_1}, \ldots, \xi_{i_r}$ and deduce the weights or errors $\omega_{i_j}$
by solving the Vandermonde system
\[ [\xi_{i_j}^{\alpha}]_{| \alpha | \leqslant d, 1 \leqslant j \leqslant r}
   (\omega_{i_j}) = (\sigma_{\alpha})_{| \alpha | \leqslant d} . \]
The points $\xi_{i_j}$ correspond to the position of the errors and
$\omega_{i_j}$ to their amplitude. By applying the border basis algorithm, we
obtain a border basis of \ the ideal of error-locator polynomials, from which
we deduce the position and amplitude of the errors.

\subsection{Sparse interpolation}

Given a sparse polynomial $h (u_1, \ldots, u_n) = \sum_{i = 1}^r \omega_i
u_1^{\gamma_{i, 1}} $ $\cdots$ $ u_n^{\gamma_{i, n}}$, which is a weighted sum of
$r$ monomials with non-zero weights $\omega_i \in \mathbbm{K}$, the problem is
to compute the exponents $(\gamma_{i, 1}, \ldots, \gamma_{i, n}) \in
\mathbbm{N}^n$ of the monomials and the weights $\omega_i$, from evaluations
of the blackbox functions $h$. \ The approach, proposed initially in
{\cite{ben-or_deterministic_1988}}, {\cite{zippel_interpolating_1990}},
consists in evaluating the function at points of the form $(\zeta_1^k, \ldots,
\zeta_n^k)$ for some values of $\zeta_1, \ldots \zeta_n \in \mathbbm{K}$ and
to apply univariate Prony-type methods or Berlekamp-Massey algorithms to the
sequence $\sigma_k = h (\zeta_1^k, \ldots, \zeta_n^k)$, for $k \in
\mathbbm{N}$. The approach can be extended to multi-index sequences
$(\sigma_{\alpha})_{\alpha \in \mathbbm{N}^n}$ by computing the terms
\[ \sigma_{\alpha} = h (\zeta_1^{\alpha_1}, \ldots, \zeta_n^{\alpha_n}) =
   \sum_{i = 1}^r \omega_i  (\zeta_1^{\gamma_{i, 1}})^{\alpha_1} \cdots
   (\zeta_n^{\gamma_{i, n}})^{\alpha_n} \]
for $\alpha = (\alpha_1, \ldots, \alpha_n) \in \mathbbm{N}^n$. \ It can also
be extended to sequences constructed from polylog functions
{\cite{mourrain_polynomial-exponential_2016}}. By applying the border basis
algorithm to the multi-index sequence $\sigma_{\alpha} = h
(\zeta_1^{\alpha_1}, \ldots, \zeta_n^{\alpha_n})$ for $| \alpha | \leqslant d$
with $d \in \mathbbm{N}$ high enough, we obtain generators of the ideal
$I_{\sigma}$ defining the points $\xi_i = (\zeta_1^{\gamma_{i, 1}}, \ldots,
\zeta_n^{\gamma_{i, n}})$ and deduce the weights $\omega_i$, $i = 1, \ldots,
r$. By computing the log of the coordinates of the points
$\xi_i$, we deduce the exponent vectors $\gamma_i = (\gamma_{i, 1}, \ldots,
\gamma_{i, n}) \in \mathbbm{N}^n$ for $i = 1, \ldots, r$.

\subsection{Tensor decomposition}

Given a homogeneous polynomial
\[ t = \sum_{\alpha_0 + \alpha_1 + \cdots + \alpha_n = d} t_{\alpha} 
   \binom{d}{\alpha} x_0^{\alpha_0} x_1^{\alpha_1} \ldots x_n^{\alpha_n} \]
of degree $d \in \mathbbm{N}$ with $t_{\alpha} \in \mathbbm{K}$,
$\binom{d}{\alpha} = \frac{d!}{\alpha_0 ! \cdots \alpha_n !}$, we want to a
decomposition of $t$ as sum of powers of linear forms:
\begin{equation}
  t = \sum_{i = 1}^r \omega_i (\xi_{i, 0} x_0 + \xi_{i, 1} x_1 + \cdots +
  \xi_{i, n} x_n)^d \label{eq:tensor}
\end{equation}
with a minimal $r$, $\omega_i \neq 0$ and $(\xi_{i, 0}, \ldots, \xi_{i, n})
\neq 0$. By a change of variables, we can assume that $\xi_{i, 0} \neq 0$ in
such a decomposition, and by dividing each linear form by $\xi_{i, 0}$ and
multiplying $\omega_i$ by $\xi_{i, 0}^d$, we can even assume that $\xi_{i, 0}
= 1$. Then by expansion of the powers of the linear forms and by
identification of the coefficients, we obtain \
\[ \sigma_{\alpha} : = t_{(d - \alpha_1 \cdots - \alpha_n, \alpha_1, \ldots,
   \alpha_n)} = \sum_{i = 1}^r \omega_i \xi_{i, 1}^{\alpha_1} \cdots \xi_{i,
   n}^{\alpha_n} = \sum_{i = 1}^r \omega_i \xi_i^{\alpha} \]
for $\alpha = (\alpha_1, \ldots, \alpha_n) \in \mathbbm{N}^n$ with $| \alpha |
\leqslant d$. We apply the border basis algorithm to this sequence, in order
to obtain generators of the ideal $I_{\sigma}$ defining the points $\xi_1,
\ldots, \xi_r \in \mathbbm{K}^n$ and providing the weights $\omega_i$. If the
number of terms $r$ is small enough compared to the number of terms
$\sigma_{\alpha}$, then the set of border relations are complete and it is
possible to compute the decomposition (\ref{eq:tensor}).

\subsection{Vanishing ideal of points}

Given a set of points $\Xi = \{ \xi_1, \ldots, \xi_r \} \subset
\mathbbm{K}^n$, we want to compute polynomials defining these points, that is,
a set of generators of the ideal of polynomials vanishing on $\Xi$. For that
purpose, we choose non-zero weights $w_i \in \mathbbm{K}$, a degree $d \in
\mathbbm{N}$ and we compute the sequence of moments{\tmabbr{}}
$\sigma_{\alpha} = \sum_{i = 1}^r \omega_i \xi^{\alpha}$ for $| \alpha |
\leqslant d$. The generating series $\sigma$ associated to these moments
define an Artinian Gorenstein algebra $\mathcal{A}_{\sigma} =\mathbbm{K}
[\tmmathbf{x}] / I_{\sigma}$, where $I_{\sigma}$ is the ideal of polynomials
vanishing $\Xi$ {\cite{mourrain_polynomial-exponential_2016}}. This ideal
$I_{\sigma}$ defines the points $\xi_i$ \ with multiplicity $1$. The
idempotents $\{ \tmmathbf{u}_i \}_{i = 1 \ldots r}$ associated to the points
$\Xi$ form a family of interpolation polynomials at these points:
$\tmmathbf{u}_i (\xi_j) = 0$ if $i \neq j$ and $\tmmathbf{u}_i (\xi_i) = 1$.
They are the common eigenvectors of the multiplication operators in
$\mathcal{A}_{\sigma}$. By applying the border basis algorithm to the sequence
$\sigma_{\alpha}$ for $| \alpha | \leqslant d$ with $d$ high enough, we obtain
generators of the ideal $I_{\sigma}$ defining the points $\xi_1, \ldots, \xi_r
\in \mathbbm{K}^n$, a basis of $\mathcal{A}_{\sigma}$ and the tables of
multiplication in this basis. By computing the eigenvectors of a generic
combination of the multiplication tables by a variable, we obtain a family of
interpolation polynomials at the roots $\Xi$.

\subsection{Benchmarks}

We present some experimentations of an implementation  of Algorithm
\ref{algo:1}\footnote{available at {\url{https://gitlab.inria.fr/mourrain/PolyExp}}}  in
the programming language  \textsc{Julia}\footnote{\url{https://julialang.org/}}. The arithmetic
operations are done in the finite field $\mathbbm{Z}/ 32003\mathbbm{Z}$. We
choose $r$ random points $\xi_i$ with $n$ coordinates in $\mathbbm{Z}/
32003\mathbbm{Z}$, take the sequence of moments $\sigma_{\alpha} = \sum_{i =
1}^r \xi_i^{\alpha}$ up for $| \alpha | \leqslant d$ \ with weights equal to
$1$. Figure  4.6 shows the timing (in sec.) to compute the border basis, checking
the validity of the recurrence relations up to degree $d$. The computation is
done on a MacOS El Capitan, 2.8 GHz Intel Core i7, 16 Go.
%\begin{figure}[ht]\caption{\label{fig:timing}}
\begin{center}
{\includegraphics[height=4.5cm]{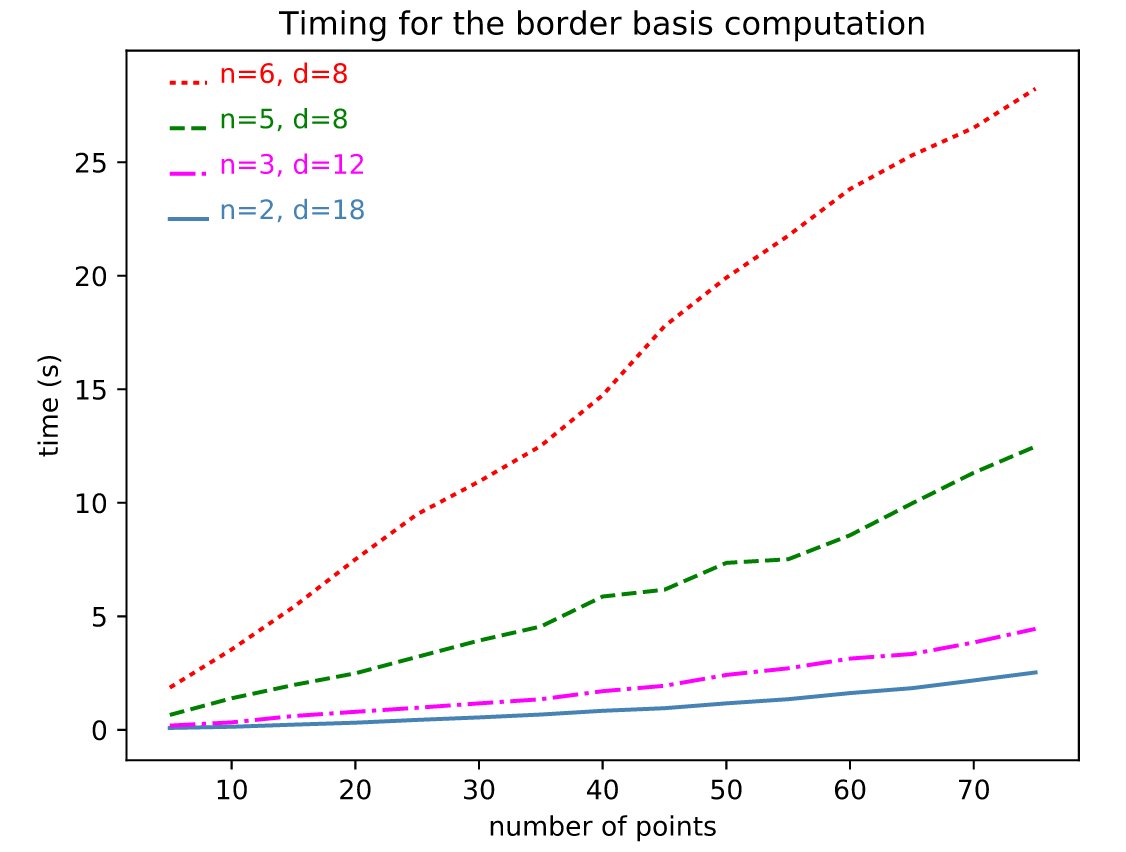} \label{fig:timing}}\\
Fig. 4.6: Vanishing ideal of random points.
 \end{center}
% \end{figure}

The timing is approximately linear in the number $r$ of points, with a slope
increasing quadratically in $n$.

{\small
%\bibliographystyle{plain}
%\bibliography{biblio}

}

%\newpage
\appendix
\section*{Examples \nopunct}

\begin{example}
  \ 
\end{example}

We consider the sequence $\sigma \in \mathbbm{K}^{\mathbbm{N}}$ such that
$\sigma_{d_1} = 1$ and $\sigma_i = 0$ for $0 \leqslant i \neq d_1 \leqslant d$
and $d_1 < d$. \

In the first step of the algorithm, we take $p_0 = 1$ and compute the first
$\gamma \in [0,\ldots,d]$ such that $\langle x^\gamma, p_1 \rangle_{\sigma}$ is not zero. This yields
$m_0 = x^{d_1}$ and $\tmmathbf{b}= [0]$, $\tmmathbf{c}= [d_1]$.

In a second step, we have $p_1 = x - \langle x, m_1 \rangle_{\sigma} p_0 = x$.
The first $\gamma \in [0,\ldots,d]\setminus\{d_{1}\}$ such that $\langle x^i, p_1 \rangle_{\sigma}$ is not zero
yields $\tmmathbf{b}= [0, 1]$, $\tmmathbf{c}= [d_1, d_1 - 1]$, $m_1 = x^{d_1 -
1}$.

We repeat this computation until $\tmmathbf{b}= [0, \ldots, d_1]$,
$\tmmathbf{c}= [d_1, d_1 - 1, \ldots, 1]$ with $m_{i} = x^{d_1 - i}$, $p_i =
x^i $ for $i = 0, \ldots, d_1$.

In the following step, we have $p_{d_1 + 1} = \tmop{proj} (x^{d_1 + 1},
\tmmathbf{p}, \tmmathbf{m}) = x^{d_1 + 1} - \langle x^{d_1 + 1}, m_1
\rangle_{\sigma} p_1 - \cdots - \langle x^{d_1 + 1}, m_{d_1} \rangle_{\sigma}
p_{d_1} = x^{d_{1 + 1}}$ such that $\langle x^{d_1 + 1}, x^j
\rangle_{\sigma} = 0$ for $0 \leqslant j \leqslant d$. The algorithm stops and
outputs $\tmmathbf{b}= [1, \ldots, x^{d_1}]$, $\tmmathbf{c}= [x^{d_1}, x^{d_1
- 1}, \ldots, 1]$, $\tmmathbf{k}= [x^{d_1 + 1}]$.\\

\begin{example}
\end{example}
We consider the function $h (u_1, u_2) = {\color{green} {\color{green} 2
+ 3} \hspace{0.25em}} \cdummy {\color{blue} 2^{u_1} 2^{u_2} {\color{green}
{\color{green} -}} 3^{u_1}}$. Its associated generating series is $\sigma
= \sum_{\alpha \in \mathbbm{N}^2} h (\alpha) \tmmathbf{z}^{\alpha} = 4 \noplus
+ 5 z_1 + 7 z_2 + 5 z_1^2 + 11 z_1 z_2 + 13 z_2^2 + \cdots$.

At the first step, we have $\mathbf{x}^{\tmmathbf{b}}= [1]$, $\tmmathbf{p}= [1]$,
$\tmmathbf{q}= \left[ \frac{1}{4} \right]$. At the second step, we compute
$\mathbf{x}^{\tmmathbf{b}}= [1, x_1, x_2]$, $\tmmathbf{p}= [1, x_1 - \frac{5}{4}, x_2 -
\frac{9}{5} x_1 - 4] = [p_1, p_{x_1}, p_{x_2}]$ and $\tmmathbf{q}= \left[
\frac{1}{4} p_1, - \frac{4}{5} p_{x_1}, \frac{5}{24} p_{x_2} \right]$. At the
next step, we obtain $\tmmathbf{k}= [], \tmmathbf{d}= [x_1^2, x_1
x_2, x_2^2]$.

\begin{eqnarray*}
  x_1 p_1 & \equiv & \frac{5}{4} p_1 + p_{x_1}\\
  x_1  \hspace{0.25em} p_{x_1} & \equiv & - \frac{5}{16} p_1 + \frac{91}{20}
  p_{x_1} - p_{x_2}\\
  x_1 p_{x_2} & \equiv & \sum_{i = 1}^3 \langle x_1 p_{x_2}, \tmmathbf{q}_i
  \rangle_{\sigma} \tmmathbf{p}_i = \frac{96}{25} p_{x_1} + \frac{1}{5}
  p_{x_2}
\end{eqnarray*}
The matrix of multiplication by $x_1$ in the basis $\tmmathbf{p}$ is
\[ M_1 = \left[ \begin{array}{ccc}
     \frac{5}{4} & - \frac{5}{16} & 0\\
     1 & \frac{91}{20} & \frac{96}{25}\\
     0 & - 1 & \frac{1}{5}
   \end{array} \right] . \]
Its eigenvalues are ${\color{blue} [\nobracket 1, 2, 3]}$ and the
corresponding matrix of eigenvectors is
\[ U \assign \left[ \begin{array}{ccc}
     \frac{1}{2} & \frac{3}{4} & - \frac{1}{4}\\
     \frac{2}{5} & - \frac{9}{5} & \frac{7}{5}\\
     - \frac{1}{2} & 1 & - \frac{1}{2}
   \end{array} \right], \]
that is, the polynomials $U (x) = [2 - \frac{1}{2}  \hspace{0.25em} x_1 -
\frac{1}{2}  \hspace{0.25em} x_2, - 1 + x_2, \frac{1}{2}  \hspace{0.25em} x_1
- \frac{1}{2}  \hspace{0.25em} x_2]$. By computing the Hankel matrix
\[ H_{\sigma}^{U, [1, x_1, x_2]} = \left[ \begin{array}{ccc}
     {\color{green} 2} & {\color{green} 3} & {\color{green} -
     1}\\
     {\color{green} {\color{green} 2 \times}}  {\color{blue} 1} &
     {\color{green} {\color{green} 3 \times}}  {\color{blue} 2} &
     {\color{green} {\color{green} - 1 \times}} {\color{blue} 3}\\
     {\color{green} {\color{green} 2 \times}}  {\color{blue} 1} &
     {\color{green} 3 \times}  {\color{blue} 2} & {\color{green}
     {\color{green} - 1 \times}} {\color{blue} 1}
   \end{array} \right] \]
we deduce the weights {\color{green} ${\color{green} {\color{green}
2, 3, - 1}}$} and the frequencies ${\color{blue} (1, 1),}$ ${\color{blue} (2,
2), (3, 1)}$, which corresponds to the decomposition $\sigma = e^{y_1 + y_2} +
3 e^{2 y_1 + 2 y_2} - e^{2 y_1 + y_2} $ associated to $h (u_1, u_2) = 2
\noplus + 3 \cdummy 2^{u_1 + u_2} - 3^{u_1}$.

\begin{example}
  \ 
\end{example}

We consider the following symmetric tensor or homogeneous polynomial:
{\small \[ \begin{array}{rl}
     \tau = & - x_0^4 - 24 \hspace{0.17em} x_0^3 x_1 - 8 \hspace{0.17em}
     x_0^3 x_2 - 60 \hspace{0.17em} x_0^2 x_1^2 - 168 \hspace{0.17em} x_0^2
     x_1 x_2 - 12 \hspace{0.17em} x_0^2 x_2^2\\
      & - 96 \hspace{0.17em} x_0 x_1^3 - 240 \hspace{0.17em} x_0 x_1^2 x_2
     - 384 \hspace{0.17em} x_0 x_1 x_2^2 + 16 \hspace{0.17em} x_0 x_2^3\\
      & - 46 \hspace{0.17em} x_1^4 - 200 \hspace{0.17em} x_1^3 x_2 - 228
     \hspace{0.17em} x_1^2 x_2^2 - 296 \hspace{0.17em} x_1 x_2^3 + 34
     \hspace{0.17em} x_2^4 .
           \end{array} \]
}%
The associated series is
\begin{eqnarray*}
  \sigma & = & - 1 - 6 \hspace{0.17em} z_1 - 2 \hspace{0.17em} z_2 - 10
  \hspace{0.17em} z_1^2 - 14 \hspace{0.17em} z_2 z_1 - 2 \hspace{0.17em}
  z_2^2\\
  &  & - 24 \hspace{0.17em} z_1^3 - 20 \hspace{0.17em} z_2 z_1^2 - 32
  \hspace{0.17em} z_2^2 z_1 + 4 \hspace{0.17em} z_2^3\\
  &  & - 46 \hspace{0.17em} z_1^4 - 50 \hspace{0.17em} z_2 z_1^3 - 38
  \hspace{0.17em} z_2^2 z_1^2 - 74 \hspace{0.17em} z_2^3 z_1 + 34
  \hspace{0.17em} z_2^4
\end{eqnarray*}
To decompose it into a sum of powers of linear forms, we apply the border
basis algorithm to the series $\sigma$. The algorithm projects successively
the monomials $1, x_1, x_2, x_1^2, x_1 x_2, x_2^2, \ldots$ onto the family of
polynomials $\tmmathbf{p}$, starting with $\tmmathbf{p}= [1]$. We obtain
$\mathbf{x}^{\tmmathbf{b}}=\tmmathbf{c}= [1, x_1, x_2]$, $\tmmathbf{p}= [1, x_1 - 6, x_2 +
\frac{1}{13} x_1 - \frac{32}{13}]$ and the border basis is
\[ \tmmathbf{k}= [{\color{red} x_1^2} - \frac{3}{2} x_1 - \frac{3}{2} x_2 + 2,
   {\color{red} x_1 x_2} - \frac{5}{2} x_1 - \frac{1}{2} x_2 + 2, {\color{red}
   x_2^2} + \frac{1}{2} x_1 - \frac{7}{2} x_2 + 2], \]
giving the projection of the border monomials $\tmmathbf{d}= [{\color{red}
x_1^2, x_1 x_2, x_2^2}]$ on the basis $\mathbf{x}^{\tmmathbf{b}}$. The decomposition of
$\tau$ is deduced from the eigenvectors of the operator of multiplication by
$x_1$:

{\scriptsize \[ M_1 = \left[ \begin{array}{ccc}
     0 & - 2 & - 2\\
     0 & \frac{1}{2} & \frac{3}{2}\\
     1 & \frac{5}{2} & \frac{3}{2}
   \end{array} \right] . \]}

Its eigenvalues are $[{\color{red} - 1, 1, 2]}$ and the eigenvectors
correspond to the polynomials
\[ \tmmathbf{u}= \left[ \begin{array}{ccc}
     \frac{1}{2}  \hspace{0.17em} x_2 - \frac{1}{2}  \hspace{0.17em} x_1 & - 2
     + \frac{3}{4}  \hspace{0.17em} x_2 + \frac{1}{4}  \hspace{0.17em} x_1 & -
     1 + \frac{1}{2}  \hspace{0.17em} x_2 + \frac{1}{2}  \hspace{0.17em} x_1
   \end{array} \right] . \]
Computing $\omega_i = \langle \sigma \mid \tmmathbf{u}_i \rangle$ and $\xi_{i,
j} = \frac{\langle \sigma \mid x_j \tmmathbf{u}_i \rangle}{\langle \sigma \mid
\tmmathbf{u}_i \rangle}$ (see {\cite{mourrain_polynomial-exponential_2016}}),
we obtain the decomposition:
\[ \tau = \left( x_0 - x_1 + 3 \hspace{0.17em} x_2 \right)^4 + (x_0 + x_1 +
   x_2)^4 - 3 \hspace{0.17em} \left( x_0 + 2 \hspace{0.17em} x_1 + 2
   \hspace{0.17em} x_2 \right)^4 . \overset{}{} \]

\begin{example}
  \ 
\end{example}
We consider the algebraic code over $\kk=\ZZ/32003\ZZ$ defined by
$$ 
C =\{ c\in \kk^{11} \mid  \sum_{i=1}^{11} c_{i}\, \xi_{i}^{\alpha} =0,\ \forall  \alpha\in \mathbbm{N}^{3}\ s.t.\  |\alpha|\leq 2\}
$$
where
$$ 
\Xi = \left[ \begin{array}{ccccccccccc}
1 & 1 & 1 &-1 &-1 &0 &0 & 2  &1 & 1 & 0\\
0 & 1 &-1 & 1 &-1 & 1 &1 &-1 & 2 &-2 & 0\\
0 & 0 &  0 & 0 &  0 &0 &1  & 1  & 1  & 1 & 1\\
\end{array}
\right]
$$
and $\xi_{i}$ is the $i^{\mathrm{th}}$ column of $\Xi$. Suppose that we receive the word
$$ 
r = [0, 3, 3, 3, 0, 0, -6, -2, 0, -1, 0]
$$
which is the sum $r= c+\omega$ of a code word $c\in C$ and an error vector $\omega\in \kk^{11}$.
We want to correct it and find the corresponding word $c$ of the code $C$.

Computing the syndromes $\sigma_{\alpha}= \sum_{i=1}^{11} r_{i} \xi_{i}^{\alpha}= \sum_{i=1}^{11} \omega_{i} \xi_{i}^{\alpha}$ for $|\alpha|\le 2$ and the corresponding (truncated) generating series, we get
$$ 
\sigma = - 2\, z_1 + z_2  + 3\,z_1\, z_2 - 2\,z_1\,z_3 - 3\,z_2^2 + z_2\,z_3.
$$
We apply the border basis algorithm to obtain error locator polynomials.
The monomials are considered in the order $\mathbf{x}^{\mathbf{a}}=[1, x_{1},x_{2},x_{3},x_{1}^{2}, x_{1}\,x_{2},\ldots,x_{3}^{2}]$. Here are the different steps, where $\nw$ denotes the new monomial introduced at each loop of the algorithm.

\noindent{}Step 1.  $\nw=1$, $\mathbf{x}^{\tmmathbf{b}}=[1]$, $\mathbf{x}^{\tmmathbf{c}}= [ x_1]$, $\tmmathbf{k}= []$.

\noindent{}Step 2. $\nw=x_{1}$, $\mathbf{x}^{\tmmathbf{b}}=[1,x_{1}]$, $\mathbf{x}^{\tmmathbf{c}}= [x_1,1]$, $\tmmathbf{k}= []$.

\noindent{}Step 3. $\nw=x_{2}$,  $\mathbf{x}^{\tmmathbf{b}}=[1,x_{1}]$, $\mathbf{x}^{\tmmathbf{c}}= [x_1,1]$, $\tmmathbf{k}= [x_{2}+\frac{1}{2} x_{1}+\frac{3}{2}]$.

\noindent{}Step 4. $\nw=x_{3}$, $\mathbf{x}^{\tmmathbf{b}}=[1,x_{1}]$, $\mathbf{x}^{\tmmathbf{c}}= [x_1,1]$, $\tmmathbf{k}= [x_{2}+\frac{1}{2} x_{1}+\frac{3}{2}, x_{3}-1]$.

The algorithm stops at this step, since the new monomial $\nw=x_{1}^{2}$ is of degree $2$ and $\nw \cdummy \mathbf{x}^{\tmmathbf{c}} \not\subset \mathbf{x}^{\mathbf{a}}$. It outputs two error locator polynomials: $x_{2}+\frac{1}{2} x_{1}+\frac{3}{2}, x_{3}-1$.

We check that only $\xi_{5}, \xi_{10}$ are roots of the error locator polynomials. We deduce the non-zero weights $\omega_{5}, \omega_{10}$ by solving the system
$ \omega_{5} \xi_{5}^{\alpha}+ \omega_{10} \xi_{10}^{\alpha} = \sigma_{\alpha}$ for $\alpha\in \{(0,0,0), (1,0,0)\}$. This yields $\omega_{5}=1, \omega_{10}=-1$, so that the code word is 
$$ 
c = [ 0, 3, 3, 3, -1, 0, -6, -2, 0, 0, 0].
$$
\end{document}